\documentclass[11pt]{amsart}

\usepackage{amssymb,amsmath,amsfonts,amsthm,mathrsfs,underscore}

\usepackage{color}
\usepackage[dvipsnames]{xcolor}
\usepackage{graphicx}
\usepackage{hyperref}
\usepackage{multicol}
\newcommand{\red}{\textcolor{red}}
\newcommand{\green}{\textcolor{OliveGreen}}
\newcommand{\blue}{\textcolor{blue}}

\usepackage{listings}
\definecolor{codegreen}{rgb}{0,0.6,0}
\definecolor{codegray}{rgb}{0.5,0.5,0.5}
\definecolor{codepurple}{rgb}{0.58,0,0.82}
\definecolor{backcolour}{rgb}{0.95,0.95,0.92}

\lstdefinestyle{mystyle}{
	backgroundcolor=\color{backcolour},   
	commentstyle=\color{codegreen},
	keywordstyle=\color{magenta},
	numberstyle=\tiny\color{codegray},
	stringstyle=\color{codepurple},
	basicstyle=\ttfamily\footnotesize,
	breakatwhitespace=false,         
	breaklines=true,                 
	captionpos=b,                    
	keepspaces=true,                 
	numbers=left,                    
	numbersep=5pt,                  
	showspaces=false,                
	showstringspaces=false,
	showtabs=false,                  
	tabsize=2
}

\theoremstyle{plain}
\newtheorem{thm}{Theorem}[section]
\newtheorem{prop}[thm]{Proposition}

\newtheorem{lemma}[thm]{Lemma}

\theoremstyle{definition}
\newtheorem{defin}[thm]{Definition}
\newtheorem{ex}[thm]{Example}
\newtheorem{rmk}[thm]{Remark}
\newtheorem{conj}[thm]{Conjecture}

\usepackage{textcomp}
\usepackage{float}

\usepackage{graphicx}

\graphicspath{ {./images/} }

\calclayout

\usepackage{cite}

\usepackage{placeins}

  \renewenvironment{thebibliography}[1]{
    \begin{oldthebibliography}{#1}
      \setlength{\parskip}{0ex}
      \setlength{\itemsep}{0ex}
  }
  {
    \end{oldthebibliography}
  }

\begin{document}

\title[Curvature sharpness and flow in weighted graphs -- Implementation ]{Bakry-\'Emery curvature sharpness and curvature flow in finite weighted graphs. II. Implementation}
	
    \author[Cushing]{David Cushing}
    \address{School of Mathematics, Statistics and Physics, Newcastle University, Newcastle upon Tyne, Great Britain}
    \email{David.Cushing1024@gmail.com}
    \author[Kamtue]{Supanat Kamtue}
    \address{Yau Mathematical Sciences Center, Tsinghua University, Beijing, China}
    \email{skamtue@tsinghua.edu.cn}
    \author[Liu]{Shiping Liu}
    \address{School of Mathematical Sciences and CAS Wu Wen-Tsun Key Laboratory of Mathematics, University of Science and Technology of China, Hefei, China}
    \email{spliu@ustc.edu.cn}
    \author[M\"unch]{Florentin M\"unch}
    \address{Max Planck Institute for Mathematics in the Sciences, Leipzig, Germany} 
    \email{muench@mis.mpg.de}
    \author[Peyerimhoff]{Norbert Peyerimhoff}
    \address{Department of Mathematical Sciences, Durham University, Durham, Great Britain}
    \email{norbert.peyerimhoff@durham.ac.uk}
    \author[Snodgrass]{Ben Snodgrass}
    \address{Department of Mathematical Sciences, Durham University, DH1 3LE, Great Britain}
    \email{hugo.b.snodgrass@durham.ac.uk}
	\date{\today}
	
	\begin{abstract}
		In this second part of a sequence of two papers, we discuss the implementation of a curvature flow on weighted graphs based on the Bakry-\'Emery calculus. This flow can be adapted to preserve the Markovian property and its limits as time goes to infinity turn out to be curvature sharp weighted graphs. 
		After reviewing some of the main results of the first paper concerned with the theoretical aspects, we present various examples (random graphs, paths, cycles, complete graphs, wedge sums and Cartesian products of complete graphs, hypercubes) and exhibit further properties of this flow. One particular aspect in our investigations is asymptotic stability and instability of curvature flow equilibria. The paper ends with a description of the available Python functions and routines available in the ancillary file. We hope that the explanations of the Python implementation via examples will help users to carry out their own curvature flow experiments.
	\end{abstract}
	
\maketitle
	
\tableofcontents
	
\section{Introduction}

This paper is concerned with computational aspects of a curvature flow on weighted graphs based on the Bakry-\'Emery calculus. This curvature flow was intoduced in our first paper \cite{CKLMPS-22}. 

A \emph{weighted graph} in this paper is a finite simple mixed combinatorial graph $G=(V,E)$ with vertex set $V$ and edge set $E = E^1 \cup E^2$ of one- and two-sided edges, together with a weighting scheme of transition rates $p_{xy} \ge 0$ for $x,y \in V$ (which can be represened by a generally non-symmetric matrix $P$ after an enumeration of the vertices). Transition rates $p_{xy}$ can only be positive if $x=y$ or if there is a one- oder two-sided edge from $x$ to $y$.
One-sided edges are denoted by ordered pairs $(x,y) \in E^1 \subset V^2$, and two-sided edges are denoted by sets $\{x,y\} \in E^2$. One- or two-sided edges $(x,y) \in E^1$ or $\{x,y\} \in E^2$ with vanishing transition rates $p_{xy}=0$ are called \emph{degenerate} and a
weighted graph $(G,P)$ is called \emph{non-degenerate} if it does not have degenerate edges.  
A weighted graph $(G,P)$ is called \emph{Markovian}, if we have $\sum_{y \in V} p_{xy} = 1$ for all $x \in V$. In this case,
$P$ is a stochastic matrix and
we can view the transition rates $p_{xy}$ as transition probabilities of a lazy random walk (with non-zero laziness if there exists a vertex $x \in V$ with $p_{xx} > 0$). Our curvature flow does not affect the underlying combinatorial graph $G$, but it changes the weighting scheme. We will focus on a version of our flow which preserves the Markovian property. In other words, starting with an initial Markovian weighted graph $(G,P_0)$, this flow will provide a family of Markovian weighting schemes $\{ P(t) \}_{t \ge 0}$ with $P(0) = P_0$, depending smoothly on the continuous time parameter $t$.
	
Before we introduce our curvature flow, we need to briefly discuss the relevant Bakry-\'Emery curvature background. This curvature notion is based on the weighted Laplacian $\Delta = \Delta_P$, acting on functions $f: V \to \mathbb{R}$ as follows:
$$ \Delta_P f(x) = \sum_{y \in V} p_{xy}(f(y)-f(x)). $$
The Laplacian gives rise to the following symmetric bilinear ``carr{\'e} du champ operators'' $\Gamma$ and $\Gamma_2$:
\begin{eqnarray*}
    2 \Gamma(f,g) &=& \Delta(fg) - f \Delta g - g \Delta f, \\
    2 \Gamma_2(f,g) &=& \Delta \Gamma(f,g) - \Gamma(f,\Delta g) - \Gamma(g,\Delta f).
\end{eqnarray*}

\subsection{Bakry-\'Emery curvature and curvature sharpness}
Bakry-Emery curvature depends on a dimension parameter $N$ and is well-defined at every vertex $x \in V$ which is \emph{not isolated}, that is, there exists another vertex $y \in V$ with $p_{xy} > 0$. For isolated vertices, there is some ambiguity how to define its curvature, and we decided to assign to such a vertex the curvature value $0$.\footnote{Another natural choice of curvature for an isolated vertex $x \in V$ would be $K_N(x) = \infty$ for all $N \in (0,\infty]$. An argument for that choice is that an isolated vertex can be viewed as a discrete analogue of a limit of round spheres with radii shrinking to $0$, whose curvatures would diverge to infinity.} The definition of Bakry-\'Emery curvature reads as follows:

\begin{defin}[Bakry-\'Emery curvature]
		The \emph{Bakry-\'Emery curvature} of a \emph{non-isolated} vertex $x \in V$ for a fixed dimension $N \in (0,\infty]$ is the supremum of all values $K \in \mathbb{R}$, satisfying the \emph{curvature-dimension inequality}
		\begin{equation} \label{eq:cd-ineq} 
			\Gamma_2(f)(x) \ge \frac{1}{N} (\Delta f(x))^2 + K\, \Gamma(f)(x) 
		\end{equation}
		for all functions $f: V \to \mathbb{R}$. We use the simplified notation $\Gamma(f) = \Gamma(f,f)$ and $\Gamma_2(f) = \Gamma_2(f,f)$.
		We denote the curvature at $x \in V$ by $K_N(x) = K_{P,N}(x)$. If $x \in V$ is isolated, that is, we have $p_{xy} = 0$ for all $y \in V \setminus \{x\}$, we set $K_N(x) = K_{P,N}(x)= 0$ for all $N \in (0,\infty]$.
\end{defin} 

This curvature notion is motivated by Bochner's identity (see, e.g., \cite[Prop. 4.15]{GHL-04}), a fundamental pointwise formula in the smooth setting of $n$-dimensional Riemannian manifolds involving gradients, Laplacians, Hessians and Ricci curvature. We refer readers to \cite{BE-84, Elw-91, Schm-99, LY-10} for the Bakry-\'Emery calculus and its application in the graph theoretic setting.

By the definition, the inequality
\begin{equation} \label{eq:cd-ineq-f} 
\Gamma_2(f)(x) \ge \frac{1}{N} (\Delta f(x))^2 + K_N(x)\, \Gamma(f)(x) 
\end{equation}
holds for every function $f$, and therefore also for the combinatorial distance function $d(x,\cdot)$. Here, $d(x,y)$ is the length of a shortest directed path from $x$ to $y$ (if there is no such path, we set $d(x,y) = \infty$). If we have equality at $x$ in \eqref{eq:cd-ineq-f} for this particular function $f=d(x,\cdot)$, we say that the vertex $x \in V$ is $N$-curvature sharp. Curvature sharpness will be particularly important in our considerations. Curvature sharpness was originally introduced in \cite[Definition 1.4]{CLP-20}. The (equivalent) definition given in this paper is inspired by \cite[Proof of Theorem 1.2]{KKRT-16}. For more details about relations between different curvature sharpness definitions see Section 3 of our first paper \cite{CKLMPS-22}.

\begin{defin}[Curvature sharpness] Let $(G,P)$ be a weighted graph and $N \in (0,\infty]$. A vertex $x \in V$ is called \emph{$N$-curvature sharp} 
if we have
\begin{equation} \label{eq:cd-up-bd} 
\Gamma_2(f)(x) = \frac{1}{N} (\Delta f(x))^2 + K_N(x)\, \Gamma(f)(x) 
\end{equation}
for the distance function $f = d(x, \cdot)$. Moreover, a vertex $x \in V$ is \emph{curvature sharp} if it is curvature sharp for some dimension $N \in (0,\infty]$. A weighted graph $(G,P)$ is called \emph{curvature sharp}, if every vertex of $G$ is curvature sharp.
\end{defin}  	
	
Note that each function $f: V \to \mathbb{R}$ with $\Gamma(f)(x) \neq 0$ gives rise to an upper curvature bound $K_{P,N}^{f}(x)$ via the inequality \eqref{eq:cd-ineq-f}. Namely, we have
\begin{equation} \label{eq:KNf-bd}
K_N(x) \le K_{P,N}^{f}(x):= \frac{1}{\Gamma(f)(x)} \left( \Gamma_2(f)(x) - \frac{1}{N}(\Delta f(x))^2  \right).
\end{equation}
A vertex $x \in V$ is therefore $N$-curvature sharp if its Bakry-\'Emery curvature $K_N(x)$ agrees with the specific upper curvature bound $K_{P,N}^{d(x,\cdot)}(x)$. We also like to mention the following monotonicity property of curvature sharpness: If $x \in V$ is $N$-curvature sharp that this vertex is also curvature sharp for any dimension $\le N$ (see \cite[Prop. 3.1]{CKLMPS-22}).

In the next subsection, we present an important reformulation of Bakry-\'Emery curvature at $x \in V$ using a specific matrix $Q(x)$, which will be important in the definition of the curvature flow.

\subsection{Reformulation of curvature via a Schur complement}

The combinatorial distance function allows us to define distance spheres and distance balls,
\begin{eqnarray*} 
S_r(x) &=& \{ z \in V: d(x,z) = r \}, \\
B_r(x) &=& \{ z \in V: d(x,z) \le r \}. \end{eqnarray*}
Let $x \in V$ be a non-isolated vertex. It turns out that the Bakry-\'Emery curvature $K_N(x)$ is determined locally, that is, can be derived solely from the information about the $2$-ball $$B_2(x) = \{x\} \cup S_1(x) \cup S_2(x) . $$ 
More precisely, denoting $S_1(x) = \{y_1,\dots,y_m\}$ and $S_2(x) = \{z_1,\dots,z_n\}$, there exist a column vector $\Delta(x)$ and a symmetric matrix $\Gamma(x)$ of size $m$ and a symmetric matrix $\Gamma_2(x)$ of size $m+n$ such that, for functions $f,g: V \to \mathbb{R}$ with $f(x) = g(x) = 0$,
\begin{eqnarray*}
  \Delta f(x) &=& \Delta(x)^\top \vec{f}_m, \\
  \Gamma(f,g)(x) &=& \vec{f}_m^\top \Gamma(x) \vec{g}_m, \\
  \Gamma_2(f,g)(x) &=& \vec{f}_{m+n}^\top \Gamma_2(x) \vec{g}_{m+n},
\end{eqnarray*}
where $\vec{f}_m = (f(y_1),\dots,f(y_m))^\top$ and $$ \vec{f}_{m+n} = (f(y_1),\dots,f(y_m),f(z_1),\dots,f(z_n))^\top, $$
and $\vec{g}_n, \vec{g}_{m+n}$, accordingly. Using the $(n+m)$-block decomposition 
$$ \Gamma_2(x) = \begin{pmatrix} \Gamma_2(x)_{S_1} & \Gamma_2(x)_{S_1,S_2} \\
\Gamma_2(x)_{S_2,S_1} & \Gamma_2(x)_{S_2} \end{pmatrix} $$
and employing the Schur complement
$$ Q(x) = \Gamma_2(x)_{S_1} - \Gamma_2(x)_{S_1,S_2} \Gamma_2(x)_{S_2}^\dagger \Gamma_2(x)_{S_2,S_1} $$
for matrices $\Gamma_2(x)$ with positive semidefinite
$\Gamma_2(x)_{S_2}$-blocks with $A^\dagger$ the pseudoinverse\footnote{For a given matrix $A \in \mathbb{R}^{N \times M}$, its pseudoinverse $A^\dagger \in \mathbb{R}^{M \times N}$ is defined by the following conditions: $A A^\dagger A = A$, $A^\dagger A A^\dagger = A^\dagger$, and $A A^\dagger \in \mathbb{R}^{N\times N}$ and $A A^\dagger \in \mathbb{R}^{M \times M}$ are both symmetric matrices.} of $A$, we can reformulate Bakry-\'Emery curvature at $x \in V$ for dimension $N \in (0,\infty]$ as follows (see Section 2.4 in \cite{CKLMPS-22}): 
\medskip

\fbox{\parbox{\textwidth}{
Let $x \in V$ be a non-isolated vertex. $K_N(x)$ is then the maximum of all $K \in \mathbb{R}$ such that
$$ Q(x) - \frac{1}{N} \Delta(x)\Delta(x)^\top - K \Gamma(x) \succeq 0, $$
where $A \succeq B$ means that $A-B$ is positive semidefinite. 
}}
\medskip

This curvature translation was motivated originally by the aim to reformulate the computation of Bakry-\'Emery curvature as an eigenvalue problem (see \cite{Sic-20,Sic-21} and \cite{CKLP-22}).

The symmetric matrix $Q(x)$ of a non-isolated vertex $x \in V$ is of size $m$ and is -- in the non-degenerate case --
closely related to another symmetric matrix $A_\infty(x)$ which, in turn, can be viewed as a discrete counterpart of the Ricci curvature tensor at a point $x \in M$ of a Riemannian manifold $(M,g)$ (see formula (1.2) and Section 7 in \cite{CKLP-22}). In the case of a Markovian weighted graph, curvature sharpness at a vertex $x \in V$ can also be alternatively expressed with the help of the matrix $Q(x)$ as follows.

\begin{thm}[see Theorem 1.3 in \cite{CKLMPS-22}] \label{thm:main}
Let $(G,P)$ be a Markovian weighted graph and $x \in V$ be a non-isolated vertex with $S_1(x) = \{y_1,\dots,y_m\}$.
Then the following statements are equivalent:
\begin{itemize}
    \item[(1)] $x$ is curvature sharp,
    \item[(2)] $x$ is curvature sharp for dimension $N=2$,
    \item[(3)] We have
    \begin{equation} \label{eq:curv-sharp-Q}
    Q(x) {\bf{1}}_m = \frac{1}{2} K_{P,\infty}^{d(x,\cdot)}(x) {\bf{p}}_x, 
    \end{equation}
    where ${\bf{p}}_x = (p_{xy_1},\dots,p_{xy_m})^\top$
    and ${\bf{1}}_m$ is the all-one column vector of size $m$.
\end{itemize}
\end{thm}
	
Note that the term $K_\infty^{d(x,\cdot)}(x)$ in \eqref{eq:curv-sharp-Q} is the upper curvature bound introduced in \eqref{eq:KNf-bd} (in the special case $N = \infty$), that is, 
$$ K_{P,\infty}^{d(x,\cdot)}(x) = \frac{\Gamma_2(d(x,\cdot))(x)}{\Gamma(d(x,\cdot))(x)}. $$ 	
	
\subsection{Curvature flow} 	
	
Let $(G,P_0)$ be a fixed initial weighted graph with $N = |V|$. For every non-isolated vertex $x \in V$, the size of the corresponding symmetric matrix $Q(x)$ agrees with the degree of the vertex $x$, that is, the number of one- and two-sided edges emanating from $x$, and the entries of $Q(x)$ are determined by the transition rates of edges of the $2$-ball $B_2(x)$. Our curvature flow
associates to this initial data a smooth matrix valued function $P: [0,\infty) \to \mathbb{R}^{N \times N}$ with $P(0) = P_0$.  The corresponding symmetric $Q$-matrices at time $t \in [0,\infty)$ depend on the weighting schemes $P(t)$, and we denote them henceforth by $Q_x(t)$ for all $x \in V$. Our curvature flow is now defined as follows.

\begin{defin}[Curvature flow]
Let $(G,P_0)$ be a finite weighted graph. The associated \emph{curvature flow} is given by the following differential equations for all non-isolated vertices $x \in V$ and all $t \ge 0$:
\begin{eqnarray}
p_{xx}'(t) &=& 0, \label{eq:laziness} \\
{\bf{p}}_x'(t) &=& - 4Q_x(t) {\bf{1}}_m + 2 C_x(t) {\bf{p}}_x(t), \label{eq:flowdiffeq}
\end{eqnarray}
where $S_1(x) = \{y_1,\dots,y_m\}$ and 
$$ {\bf{p}}_x(t) = (p_{xy_1}(t),\dots,p_{xy_m}(t))^\top. $$
In the case of an isolated vertex $x \in V$, its curvature flow is given by the simple equation $p_{xx}'(t) = 0$, that is $p_{xx}(t)$ is a constant function in $t$.
\end{defin} 

We note that the curvature flow equation \eqref{eq:laziness} guarantees that the diagonal entries of the weighting scheme do not change. The functions $C_x(t)$ in the curvature flow equation \eqref{eq:flowdiffeq} play the role of a normalisation since, for the choice $C_x \equiv 0$, various transition rates will be unbounded as the time parameter $t \ge 0$ increases. Note that in the smooth case of closed Riemannian manifolds $(M,g_0)$, a suitable normalization leads to volume preservance of $(M,g_t)$ under the Ricci curvature flow. Our aim is to preserve the Markovian property, and it was shown in our first paper that the curvature flow preserves this property if we choose
the normalization functions
\begin{equation} \label{eq:CxMark}
C_x(t) = K_{P(t),\infty}^{d(x,\cdot)}(x).
\end{equation}
Let us give the explicit formulas for the curvature flow equation \eqref{eq:flowdiffeq} for this particular choice of $C_x(t)$, where $y,y',y''$ always represent vertices in $S_1(x)$ (see \cite[formula (66)]{CKLMPS-22}:  
\begin{multline} \label{eq:flowdiffeq-explicit}
   p_{xy}'(t) = \\ p_{xy}(t)\left( -4p_{yx}(t) - 2 \sum_{y' \neq y} p_{yy'}(t) + \frac{4}{D_x} \sum_{y'} p_{xy'}(t)p_{y'x}(t) + \frac{1}{D_x}\sum_{y',y''} p_{xy'}(t)p_{y'y''}(t) - p_{yy}(t) \right) \\+ \underbrace{\sum_{y' \neq y} p_{xy'}(t)p_{y'y}(t)}_{(*)}.
\end{multline}
Here we use $D_x = \sum_{y'} p_{xy}(t) = 1 - p_{xx}(t)$.
note that \eqref{eq:laziness} guarantees that $D_x \le 1$ is independent of the time parameter $t$.

The following theorem collects some fundamental properties of the normalized curvature flow.

\begin{thm}[see Theorem 1.5 and Prop. 1.6 in \cite{CKLMPS-22}] \label{thm:curvflowprops}
   Let $(G,P_0)$ be a Markovian weighted graph. Then the curvature flow $(G,P(t))_{t \ge 0}$ associated to $(G,P_0)$ with normalization \eqref{eq:CxMark} is well defined for all $t \ge 0$ and preserves the Markovian property.
   If $(G,P_0)$ is non-degenerate, then $(G,P(t))$ is also non-degenerate for all $t \ge 0$. Moreover, if the flow converges for $t \to \infty$ to $P^\infty = \lim_{t \to \infty} P(t)$, then the weighted graph $(G,P^\infty)$ is curvature sharp. 
\end{thm}
 
We like to emphasize that, even in the Markovian case, a flow limit $P^\infty = \lim_{t \to \infty} P(t)$
of a non-degenerate weighted graph $(G,P_0)$ is in most cases no longer non-degenerate, despite the fact that all weighting schemes $P(t)$ for finite $t \ge 0$ are non-degenerate. In other words, some transition probabilities converge to zero under our normalized curvature flow, as time tends to infinity.
	
\section{Curvature flow examples}
\label{sec:curv-flow-ex}

In this section, we investigate normalized curvature flows on some unmixed combinatorial graphs $G = (V,E)$. By \emph{unmixed} we mean that $G$ does not have one-sided edges. 
We assume in all examples and statements in this section that our graphs $G=(V,E)$ are finite, simple, unmixed and connected and that our initial weighting schemes $P_0= P(0)= (p_{xy}(0))_{x,y \in V}$ are non-degenerate Markovian without laziness (even if we do not mention this). Curvature flow limits $(G,P^\infty)$ with $P^\infty = \lim_{t \to \infty} P(t)$ are necessarily curvature sharp by Theorem \ref{thm:curvflowprops} above. We do not know of any initial Markovian weighted graph which does not converge as $t \to \infty$. Moreover, we know that every finite connected graph with at least two vertices admits many curvature sharp Markovian weighting schemes without laziness (see \cite[Theorem 1.10]{CKLMPS-22}).
These facts give rise to the following conjecture.

\begin{conj}[see Conjecture 1.7 in \cite{CKLMPS-22}] The curvature flow $(G,P(t))_{t \ge 0}$ with normalization 
\eqref{eq:CxMark} converges for any initial condition $(G,P_0)$ as $t \to \infty$.
\end{conj}
 
Let us now address some practical aspects of the curvature flow implementation. The solution of the curvature flow is computed numerically by the Runge-Kutta (RK4) method, which is based on a time discretization with time increments $dt > 0$. In the following examples we choose the step sizes $dt = 0.1$ and $dt = 0.3$. In order to distinguish between the theoretical curvature flow and its implementation, we refer to the latter as the \emph{numerical curvature flow}.
Since a numerical curvature flow cannot run forever, a suitable numerical convergence criterion needs to be introduced. Our convergence criterion is based on the parameter ${\rm{lim}}_{\rm{tolerance}} > 0$. We say that a numerical curvature flow solution $(P(t))_{t \ge 0}$ has converged numerically at time $t$ (with respect to the parameter ${\rm{lim}}_{\rm{tolerance}}$), if all the entries of $P(t)$ differ from the corresponding entries of $P(t+10)$ and $P(t+20)$ by less than ${\rm{lim}}_{\rm{tolerance}}$. The numerical flow limit is then defined to be $P(t)$. In all examples to follow we set ${\rm{lim}}_{\rm{tolerance}} = 0.001$.

In this section, we are particularly interested in numerical flow limits which are not \emph{numerically totally degenerate}. An unmixed weighted graph $(G,P)$ is called \emph{numerically totally degenerate} if there are no edges with numerical non-zero transition rates in both directions, where we consider a transition rate $p_{xy}$ as numerically non-zero (with respect to a parameter ${\rm{threshold}} > 0$), if and only if $p_{xy} \ge {\rm{threshold}}$. In all examples to follow we set ${\rm{threshold}} = 0.001$.  
 
\subsection{Random graphs}

In this subsection we investigate the numerical curvature flow for random weighted graphs $(G,P_0)$ with vertex set $V$. The edge set $E$ of $G$ is generated by an Erd\"os-Renyi process, that is, any pair of vertices is independently and randomly connected by a two-sided edge with a probability $p \in [0,1]$. Similarly, we choose a random initial weigthing scheme $P_0$ with the property that all non-zero transition rates $p_{vv'}(0)$ lie in the interval $[{\rm{threshold}},1]$, for some positive parameter ${\rm{threshold}} > 0$. We are interested in properties of numerical flow limits of these graphs. If the parameter $p > 0$ in the Erd\"os-Renyi process is too small, these flow limits are always numerically totally degenerate. A reasonable choice to obtain not numerically totally degenerate flow limits for such random graphs with, say, $10$ vertices in roughly half of the cases, is $p = 0.7$. 

\begin{ex}[A random graph with $10$ vertices] \label{ex:random-graph}
Let $(G,P_0)$ be the unmixed weighted Markovian graph with
vertex set $V = \{v_0,\dots,v_9\}$ and
$$ P_0 = \begin{pmatrix}
0&  0.1&  0.08& 0.17& 0&  0.28& 0.21& 0.08& 0&  0.08\\ 
0.08& 0&  0&  0.16& 0&  0.2&  0.07& 0.3&  0.04& 0.15\\ 
0.27& 0&  0&  0&  0&  0&  0.3&  0&  0&  0.43\\ 
0.02& 0.19& 0&  0&  0.17& 0.17& 0.11& 0.34& 0&  0\\
0&  0&  0&  1&  0&  0&  0&  0&  0&  0\\
0.04& 0.21& 0&  0.41& 0&  0&  0.34& 0&  0&  0\\ 
0.06& 0.29& 0.14& 0.12& 0&  0.3&  0&  0.09& 0&  0\\
0.08& 0.31& 0&  0.19& 0&  0&  0.23& 0&  0.19& 0\\ 
0&  0.13& 0&  0&  0&  0&  0&  0.25& 0&  0.62\\ 
0.1&  0.33& 0.38& 0&  0&  0&  0&  0&  0.19& 0 \end{pmatrix} $$
This randomly generated initial graph is illustrated on the left hand side of Figure \ref{fig:random-graph}. The numerical curvature flow of $(G,P_0)$ has numerical convergence time $t_{\rm{max}} = 20.7$ (with respect to ${\rm{lim}}_{\rm{tolerance}} = 0.001$). The numerical flow limit $(G,P(t_{\rm{max}}))$ is presented on the right hand side of Figure \ref{fig:random-graph}. Let us briefly explain the illustration of the edges of this flow limit: Edges with numerical non-zero transition rates in both directions are displayed in green. Edges with only one-sided numerical non-zero transition rates are displayed as dashed red lines with arrows. Edges whose transition rates shrink in both directions numerically to zero under the curvature flow are displayed as dotted black lines. The corresponding non-zero transition rates are written along these edges. The vertices of the green edges of the flow limit in Figure \ref{fig:random-graph} are given by 
$$ W = \{ v_0, v_1, v_3, v_5, v_6, v_7 \}. $$
Since there are no numerical non-zero transition rates from these vertices to the other vertices $v_2, v_4, v_8, v_9 \in V \setminus W$, the vertex set $W$ together with the green edges and their transition rates represent a highly connected non-degenerate Markovian weigthed subgraph $(G_W,P_W)$. The combinatorial graph $G_W$ with vertex set $W$ can be viewed as a double cone over the complete graph $K_4$ of the four vertices $v_0, v_1, v_3, v_6$ with the vertices $v_5,v_7$ as its cone tips. The transition rates of the weighting scheme $P_W$ towards all vertices in $K_4$ are $0.25 = 1/4$, all transition rates towards $v_5$ are $x=0.05$ and all transition rates towards $v_7$ are $y=0.2$. Note that such a weighted double cone over $K_4$ with these transition rates for any choice of $x,y > 0$ satisfying $x+y = 1/4$ is curvature sharp (this follows readily from the geometric criterion in \cite[Theorem 3.15]{CKLMPS-22}, since the weighting scheme is volume homogeneous in all vertices and reversible with $\pi(v_5) = 4x/5$, $\pi(v_7) = 4y/5$ and $\pi\vert_{K_4} \equiv 1/5$). Therefore, not only the flow limit
itself is curvature sharp by Theorem \ref{thm:curvflowprops} but also the highly connected non-degenerate weighted subgraph $(G_W,P_W)$. 

Figure \ref{fig:random-graph-trans-rates} presents the transition rates of the vertices $v_3, v_7$ under the curvature flow as functions over the interval $[0,t_{\rm{max}}]$. While most transition rates converge to strictly positive limits, the transition rates $p_{v_3v_4}(t)$ and $p_{v_8v_9}(t)$ shrink to zero. Consequently, the corresponding edges on the right hand side of Figure \ref{fig:random-graph} are represented by a dashed red line and a black dotted line, respectively.

Let us finally, consider the curvatures $t \mapsto K_{P(t),N}(v_j)$ of the vertices $v_j$ under the curvature flow. We focus exemplary on the vertices $v_1$ and $v_9$ and the dimension parameter $N = \infty$. Figure \ref{fig:random-graph-curvatures} presents the $\infty$-curvature (in blue) and upper curvature bound $K_{P(t),\infty}^{d(v,\cdot)}(v)$ (in orange) of $v \in \{v_1,v_9\}$ as functions over the interval $[0,t_{\rm{max}}]$. Note that the absence of laziness implies that the upper curvature bounds $K_{P(t),\infty}^{d(v,\cdot)}(v)$ are all $\ge 0$ (see \cite[(19)]{CKLMPS-22}). For both vertices, the curvature and upper curvature bound functions are numerically asymptotic as $t \to t_{\rm{max}}$, indicating that $v_1$ and $v_9$ of the flow limit are $\infty$-curvature sharp. (In fact, all vertices in this example are $\infty$-curvature sharp with respect to the flow limit.) This is not always the case, but Theorem \ref{thm:main} confirms that all vertices of a flow limit $(G,P^\infty)$ are at least $N$-curvature sharp for dimension $N=2$. Moreover, the initial and final $\infty$-curvatures of all vertices under the curvature flow are presented in Table \ref{tab:random-graph-curvatures}. The final curvatures of all vertices in $K_4$ assume the highest values $0.875$, followed by the
final curvature values $0.773$ and $0.476$ of the vertices $v_7$ and $v_5$, respectively. All other vertices in $V \setminus V_0$ have much lower final curvatures with values $\le 0.125$.

\begin{figure}[h]
\includegraphics[width=0.49\textwidth]{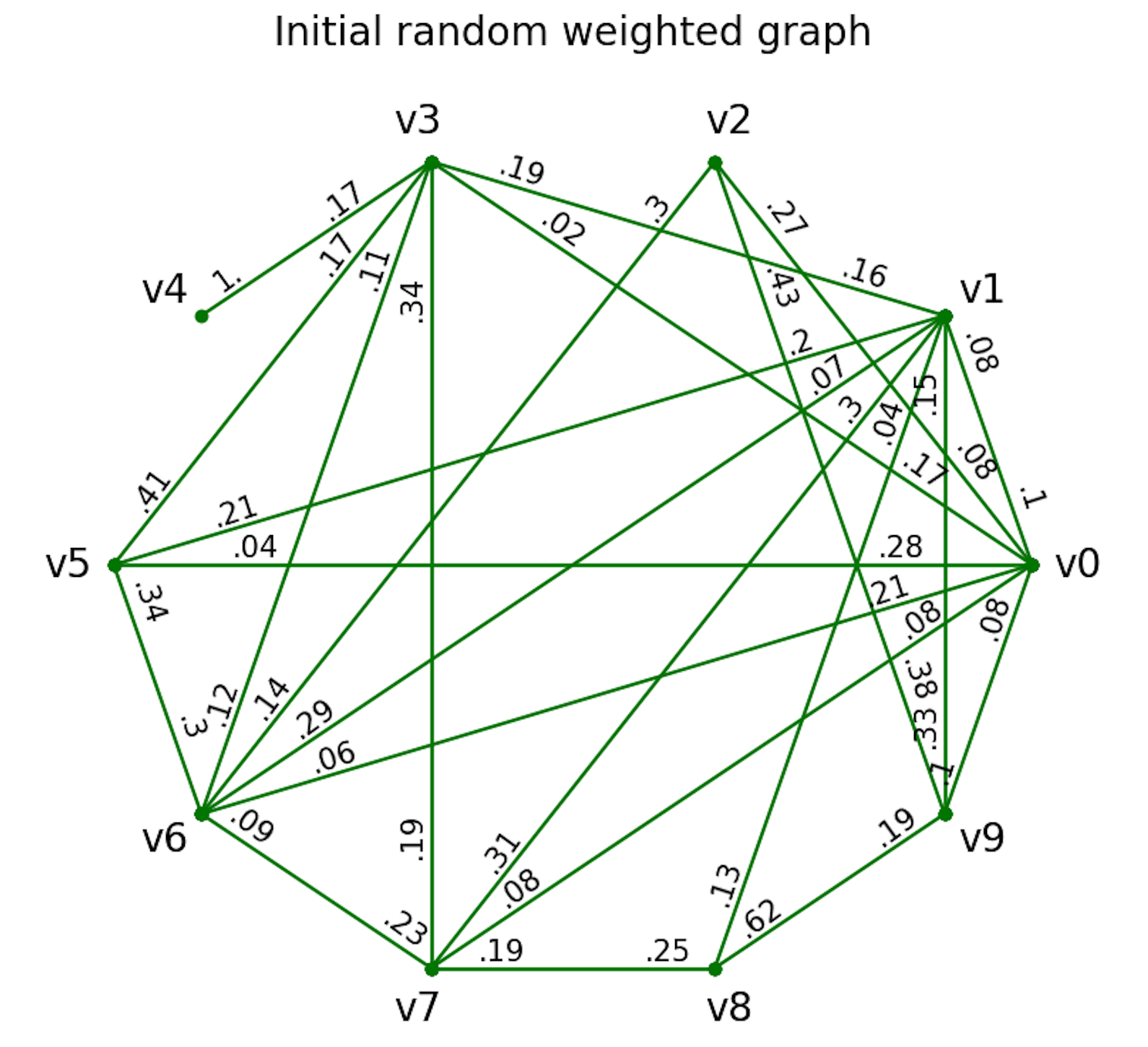}
\includegraphics[width=0.49\textwidth]{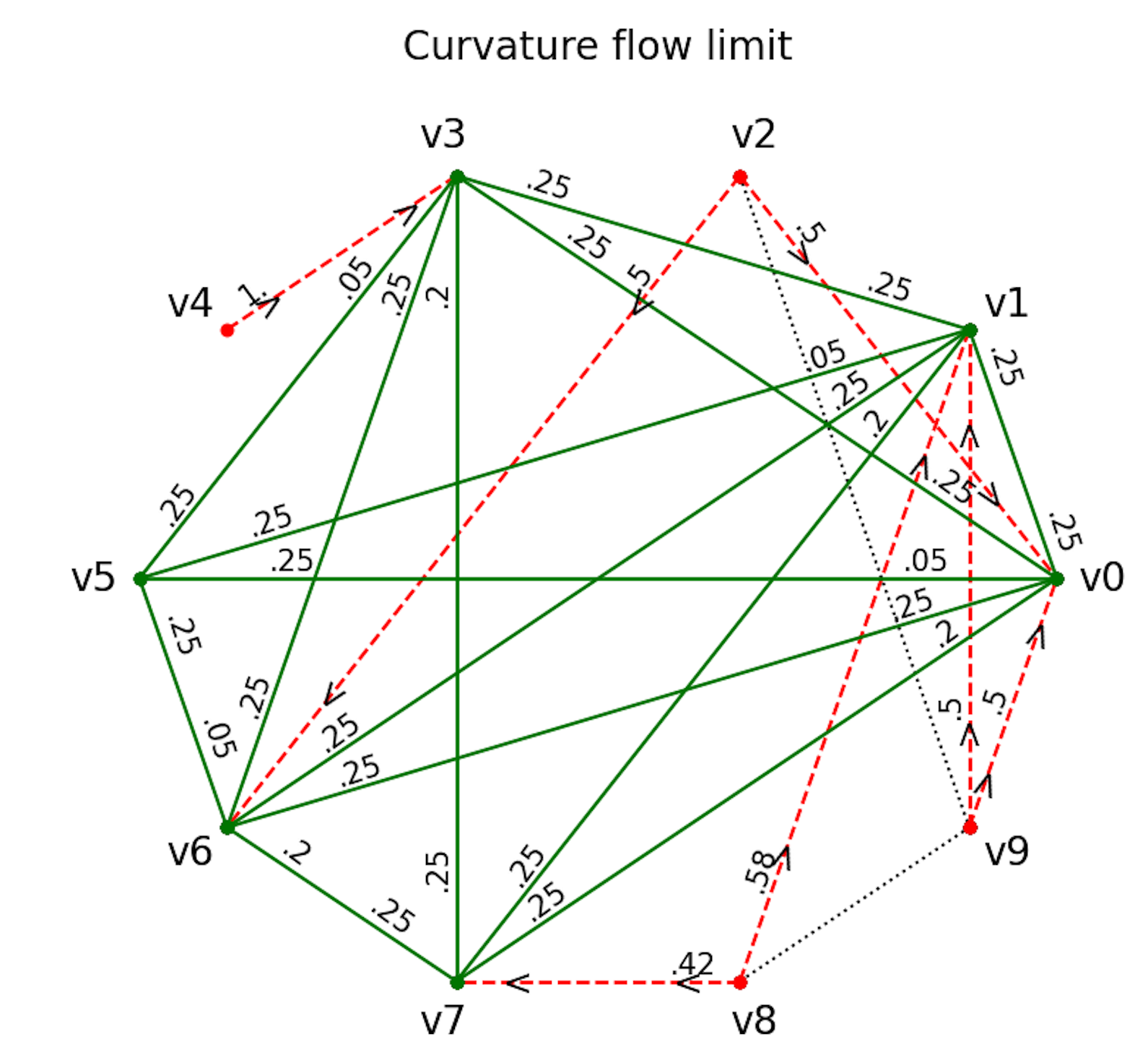}
\caption{Curvature flow of a random graph with $10$ vertices with initial weighting scheme $P_0$ (left hand side) and final weighting scheme $P(t_{\rm{max}})$ (right hand side)}
\label{fig:random-graph}
\end{figure}

\begin{figure}[h]
\includegraphics[width=0.49\textwidth]{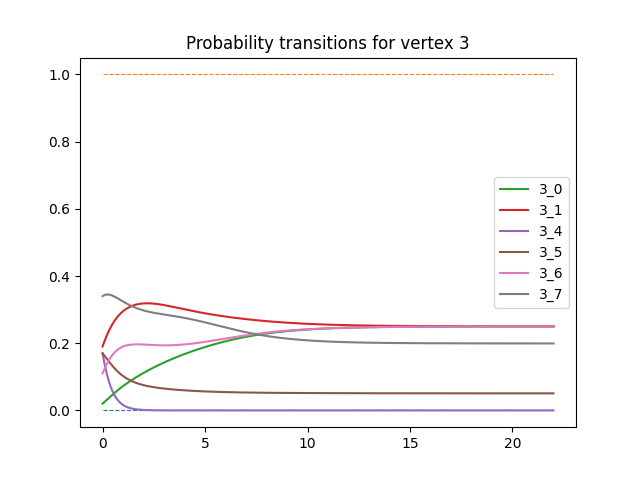}
\includegraphics[width=0.49\textwidth]{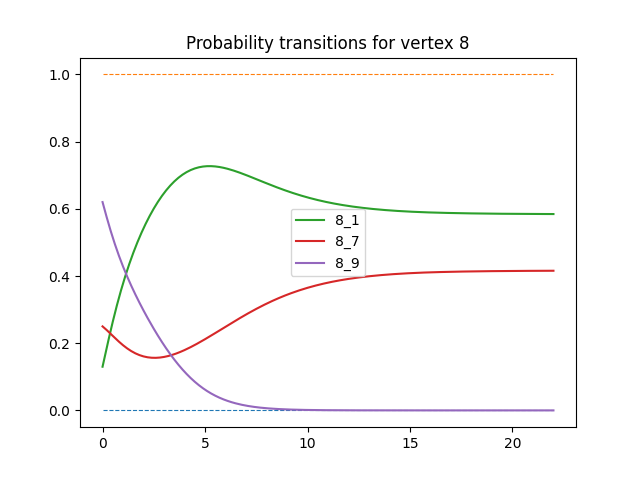}
\caption{Transition rates of vertices $v_3$ and $v_8$ of a random graph with $10$ vertices under the curvature flow}
\label{fig:random-graph-trans-rates}
\end{figure}

\begin{figure}[h]
\includegraphics[width=0.49\textwidth]{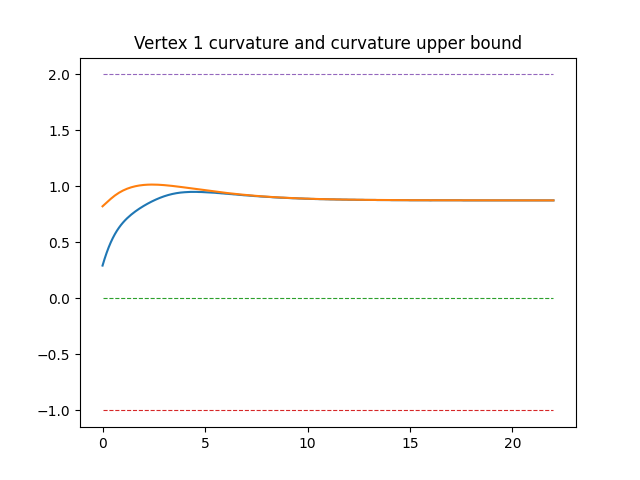}
\includegraphics[width=0.49\textwidth]{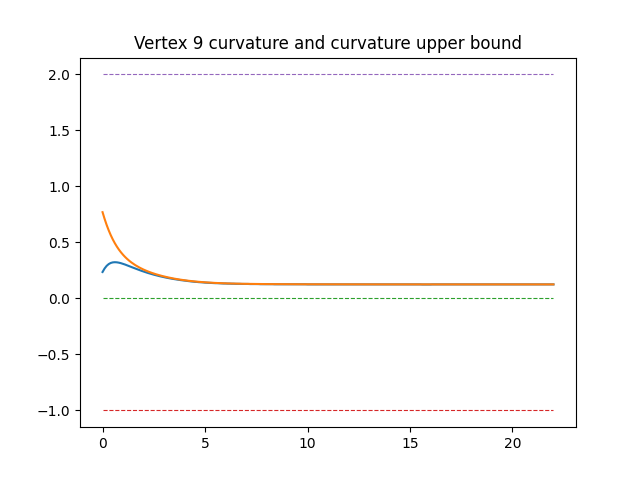}
\caption{Curvatures (blue) and upper curvature bounds (orange) of vertices $v_1$ and $v_9$ of a random graph with $10$ vertices for the dimension $\infty$ under the curvature flow}
\label{fig:random-graph-curvatures}
\end{figure}

\begin{table}[h]
\begin{minipage}[h]{0.49\textwidth}
\begin{center}
\begin{tabular}{l|ll}
$j$ & $K_{P(0),\infty}(v_j)$ & $K_{P(t_{\rm{max}}),\infty}(v_j)$ \\ \hline\\[-0.4cm]
$0$ & $0.406$ & $0.875$ \\
$1$ & $0.293$ & $0.875$ \\
$2$ & $0.168$ & $0.125$ \\
$3$ & $0.346$ & $ 0.875$ \\
$4$ & $0.34$ & $0$
\end{tabular}
\end{center}
\medskip
\end{minipage}
\begin{minipage}[h]{0.49\textwidth}
\begin{center}
\begin{tabular}{l|ll}
$j$ & $K_{P(0),\infty}(v_j)$ & $K_{P(t_{\rm{max}}),\infty}(v_j)$ \\ \hline\\[-0.4cm]
$5$ & $0.527$ & $0.476$ \\
$6$ & $0.404$ & $0.875$ \\
$7$ & $0.202$ & $0.773$ \\
$8$ & $0.246$ & $0.11$ \\
$9$ & $0.236$ & $0.125$
\end{tabular}
\end{center}
\medskip
\end{minipage}
\caption{Vertex curvatures for the dimension $\infty$ of a random graph with $10$ vertices at the beginning and the end of the curvature flow}
\label{tab:random-graph-curvatures}
\end{table}
\end{ex}

\FloatBarrier

Example \ref{ex:random-graph} and its illustrations was generated by the following code. 

\begin{lstlisting}[language=Python]{}
dt = 0.1                  # time increment in the Runge-Kutta (RK4) algorithm 
stoch_corr = False        # automatic stochastic correction in the RK4 algorithm
norm_tolerance = 0.001    # threshold to apply stochastic correction
threshold = 0.001         # threshold to consider a number numerally as zero
lim_tolerance = 0.001     # threshold to define flow convergence
t_lim = 10000             # maximal flow time 
p = 0.7                   # Erdoes-Renyi probability parameter
k = 1                     # time multiplier for consecutive curvature computations
is_Markov = True          # curvatures are w.r.t. a Markovian weighting scheme
N = inf                   # dimension parameter for curvature computations
laziness = False          # Boolean in the generation of random weighting schemes   

A = rand_adj_mat(10, 0.7, False) 
P_0 = randomizer(A, threshold, laziness)

limit = norm_curv_flow_lim(A, P_0, dt, stoch_corr, norm_tolerance, lim_tolerance, t_lim)

t_max=limit[1]            # limit[1] is the convergence time
                          # as calculated by norm_curv_flow_lim

flow = norm_curv_flow(A, P_0, t_max, dt, stoch_corr, norm_tolerance)

display_weighted_graph(A, P_0, "Initial random weighted graph", threshold)
display_weighted_graph(A, limit[0], "Curvature flow limit", threshold)
display_trans_rates(A, flow, dt, [3, 8])

curvs = calc_curvatures(A, flow, N, k)
curv_bound = calc_curv_upper_bound(A, flow, N, k)
display_curvatures(curvs, dt, is_Markov, N, k, curv_bound, [1, 9])
\end{lstlisting}

Let us explain the commands of this program in detail. After initializing various parameters in lines 1-11, a random graph $G=(V,E)$ is generated in line 13 via an Erd\"os-Renyi process with probability $p=0.7$. A corresponding random non-degenerate Markovian weigthing scheme without laziness is generated in line 14, with each non-zero transition rate satisfying $p_{vv'}(0) \in [{\rm{threshold}},1]$. The numerical convergence time $t_{\rm{max}} \ge 0$  is determined in lines 16-17. In most cases, the convergence time does not exceed $100$. (If convergence is not achieved by $t_{\rm{lim}} = 10000$, the curvature flow computation stops at that time and notifies the user.) The numerical curvature flow on the interval $[0,t_{\rm{max}}]$ is solved again in line 21. Initial and final weighting schemes are displayed by the commands in lines 23 and 24, respectively, providing the illustrations given in Figure \ref{fig:random-graph}. The transition rates of the vertices $v_3$ and $v_8$ are displayed by the command in line 25, providing the illustrations given in Figure \ref{fig:random-graph-trans-rates} The curvatures and upper curvature bounds at time steps $j \cdot \frac{k \cdot t_{\rm{max}}}{dt}$, $j=0,1,\dots$, are computed in lines 27 and 28, respectively, with the choice $k=1$. They are displayed for the vertices $v_1$ and $v_9$ via the command in line 29, providing the illustrations given in Figure \ref{fig:random-graph-curvatures}. 

\medskip

When running this program, users may be faced with the following message:

\medskip

\texttt{`norm_tolerance' has been exceeded at one or more vertices, at time t = ... Would you like to:} \\
\texttt{A = Stop calculation and return list of P-matrices so far} \\
\texttt{B = Apply manual normalization now, and apply it again when necessary without asking (you will still be notified when it is applied)} \\
\texttt{C = Apply manual normalization now, and ask again before reapplying it} \\
\texttt{Please enter A, B or C here:}

\medskip

The reason behind this message is the following. The computation of the numerical curvature flow is based on a time discretization. Therefore, the solution will increasingly depart from the Markovian property after each time increment $dt = 0.1$. If the sum of entries of one of the rows of $P(t)$ at time $t$ differs from one by more than ${\rm{norm}}_{\rm{tolerance}}=0.001$, the program informs the user that a normalization of the weighting scheme is needed for the continuation of the flow calculations. After choosing the option \texttt{'B'}, the program will continue with its flow calculations without further interruptions, and the user is simply notified about the times at which the program applies further artificial normalizations of the transition rates. The user can suppress this message entirely by changing line 2 of the program into "\texttt{stoch_corr = True}", in which case the program applies stochastic corrections automatically, each time with the message

\medskip

\texttt{Transition rates have been artificially normalized at time t = ...}

\medskip

The numerical observations of Example \ref{ex:random-graph} suggest that similar properties may also hold in general in the theoretical setting. Firstly, we call an edge $\{x,y\} \in E$ in an unmixed weighted graph $(G,P)$ \emph{non-degenerate} if $p_{xy}, p_{yx} > 0$. An unmixed weighted graph $(G,P)$ is called \emph{totally degenerate} if it does not have any non-degenerate edges. Secondly, we denote the vertices of all non-degenerate edges of a flow limit $(G,P^\infty)$ by $W \subset V$, and $G_W$ denotes the subgraph of $G$ consisting of the vertices $W$ and all non-degenerate edges of $(G,P^\infty)$. Moreover, $P_W$ denotes the restriction of the weigthing scheme $P^\infty$ to the vertex set $W$. We call $G_W$ the non-degenerate subgraph of $(G,P^\infty)$ and we conjecture the following: 

\begin{conj}
If the normalized curvature flow of a nondegenerate unmixed Markovian weighted graph $(G,P_0)$ converges to a not totally degenerate limit $(G,P^\infty)$, then the non-degenerate subgraph $G_W$ coincides with the induced subgraph (of $G$) of the subset $W \subset V$ and all transition rates from $W$ to $V \setminus W$ are zero. $(G_W,P_W)$ is a non-degenerate Markovian weighted graph which is itself curvature sharp. 
\end{conj}

Figuratively speaking, the curvature flow converges towards the non-degenerate Markovian subgraph $(G_W,P_W)$ consisting of highly connected components. Moreover, each vertex of $V \setminus W$ is usually connected to the set $W$ by a sequence of edges with one-sided non-zero transition rates pointing towards the set $W$, and we generally expect that the $\infty$-curvature values of the vertex set $W$ in $(G,P^\infty)$ are significantly larger than the $\infty$-curvature values of the set $V \setminus W$ in $(G,P^\infty)$.

\medskip

In Example \ref{ex:random-graph}, the weighted Markovian subgraph $(G_W,P_W)$ has only one connected component, but we will see in the next subsection in the case of paths and cycles that $(G_W,P_W)$ may be composed of more than just one connected component. 

\subsection{Paths and cycles} \label{subsec:paths-cyles}
Let $G=(V,E)$ be a path of length $N \ge 2$, that is $V = \{v_0,\dots,v_{N-1} \}$ with a two-sided edge between $v_i $ and $v_j$ if and only if $|i-j|=1$. If $N=2$, $G$ is a trivial case of a star graph and any Markovian weigthing scheme satisfying $p_{01}=p_{21}=1$ and $p_{10}+p_{12}=1$ is curvature sharp (see \cite[Example 4.3]{CKLMPS-22}. For that reason we consider only paths of lengths $N \ge 3$. A cycle of length $3$ is the complete graph $K_3$, and a full list of all curvature sharp weighting schemes was given in \cite[Prop. 1.9]{CKLMPS-22}, so we consider only cycles of length $N \ge 4$. The following example provides some insights into some features of curvature flow limits of weighted paths and cycles.

\begin{ex}[A path and a cycle of length $12$] \label{ex:path-cycle} 

Figure \ref{fig:path-cycle-graph} presents numerical curvature flow limits of a Markovian weighted path with vertices $12$ vertices (left hand side) and of a weighed cycle with $12$ vertices (right hand side). The non-zero transition rates of the initial weighting scheme $P_0 = (p_{ji})_{0\le i,j \le 11}$ for the path limit in Figure \ref{fig:path-cycle-graph} were chosen as follows:

\medskip

\begin{center}
\begin{tabular}{lllllllllll}
$p_{0,1}$ & $p_{1,2}$ & $p_{2,3}$ & $ p_{3,4}$
& $p_{4,5}$ & $p_{5,6}$ & $p_{6,7}$ & $p_{7,8}$ & $p_{8,9}$ & $p_{9,10}$ & $p_{10,11}$ \\ 
$1$ & $.025$ & $0.72$ & $0.46$ & $0.23$ & $019$ & $0.84$ & $0.71$ & $0.62$ & $0.9$ & $0.55$ \\[.2cm]
$p_{1,0}$ & $p_{2,1}$ & $p_{3,2}$ & $ p_{4,3}$
& $p_{5,4}$ & $p_{6,5}$ & $p_{7,6}$ & $p_{8,7}$ & $p_{9,8}$ & $p_{10,9}$ & $p_{11,10}$ \\ $0.75$ & $0.28$ & $0.54$ & $0.77$ & $0.81$ & $0.16$ & $0.29$ & $0.38$ & $0.1$ & $0.45$ & $1$ \\
\end{tabular}
\end{center}

\medskip

The non-degenerate subgraph $G_W$ of the path limit consists of $W = \{v_1,v_2,v_3,v_9,v_{10},v_{11}\}$ as its vertex set together with the four green edges. Moreover, $G_W$ has two paths of length $2$ as its connected components. For each vertex $v \in V \backslash W$ there exists a directed path to one of the components of $G_W$ via a sequence of one-sided transition 
rates. For example the vertices $v_5,v_6,v_7,v_8$ are connected to the vertex $v_9 \in W$ via such one-sided paths and $v_4,v_5$ are connected to the vertex $v_3$ via such one-sided paths. 

The cycle limit on the right hand side of Figure \ref{fig:path-cycle-graph} is totally degenerate with no green edges, and all one-sided non-zero transition rates are oriented in a clockwise direction. Experiments show that Figure \ref{fig:path-cycle-graph} exhibit generic limit properties: flow limits of weighted paths are never totally degenerate and their non-degenerate subgraphs $G_W$ consist of disjoint paths of length $\le 2$. Such types of limits appear also in the case of weighted cycles. However, in contrast to the path case, sometimes a cycle limit is totally degenerate with all its one-sided transition rates oriented either clockwise or anti-clockwise. The code for running the curvature flow for paths and cycles of length $n$ reads as follows:

\begin{lstlisting}[language=Python]{}
N = 12
A = path(N)
# A = cycle(N)
P = randomizer(A)
limit = norm_curv_flow_lim(A, P)[0]
display_weighted_graph(A, P, "Initial weighted graph")
display_weighted_graph(A, limit, "Curvature flow limit")
\end{lstlisting}

\medskip

\begin{figure}[h]
\includegraphics[width=0.49\textwidth]{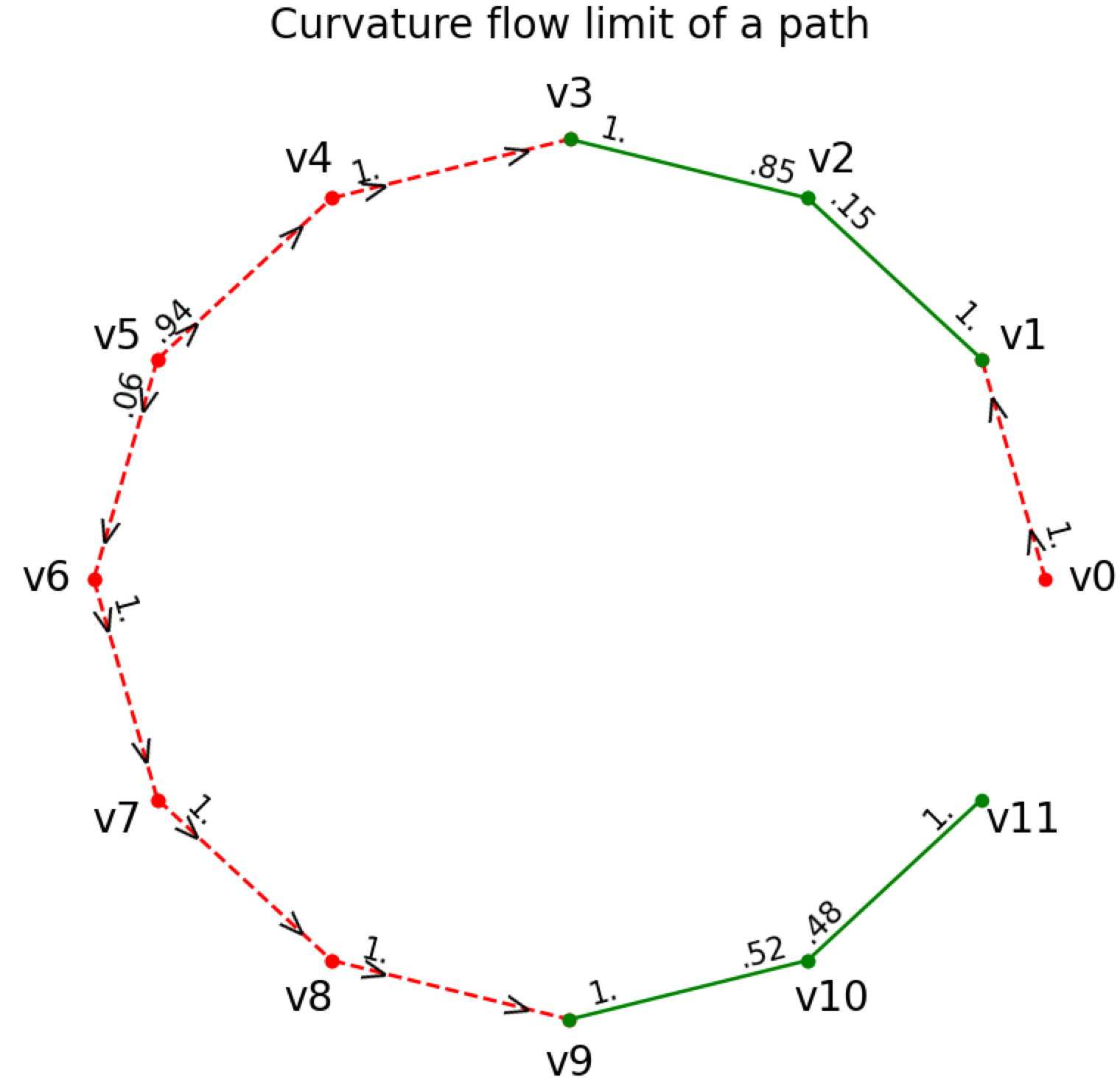}
\includegraphics[width=0.49\textwidth]{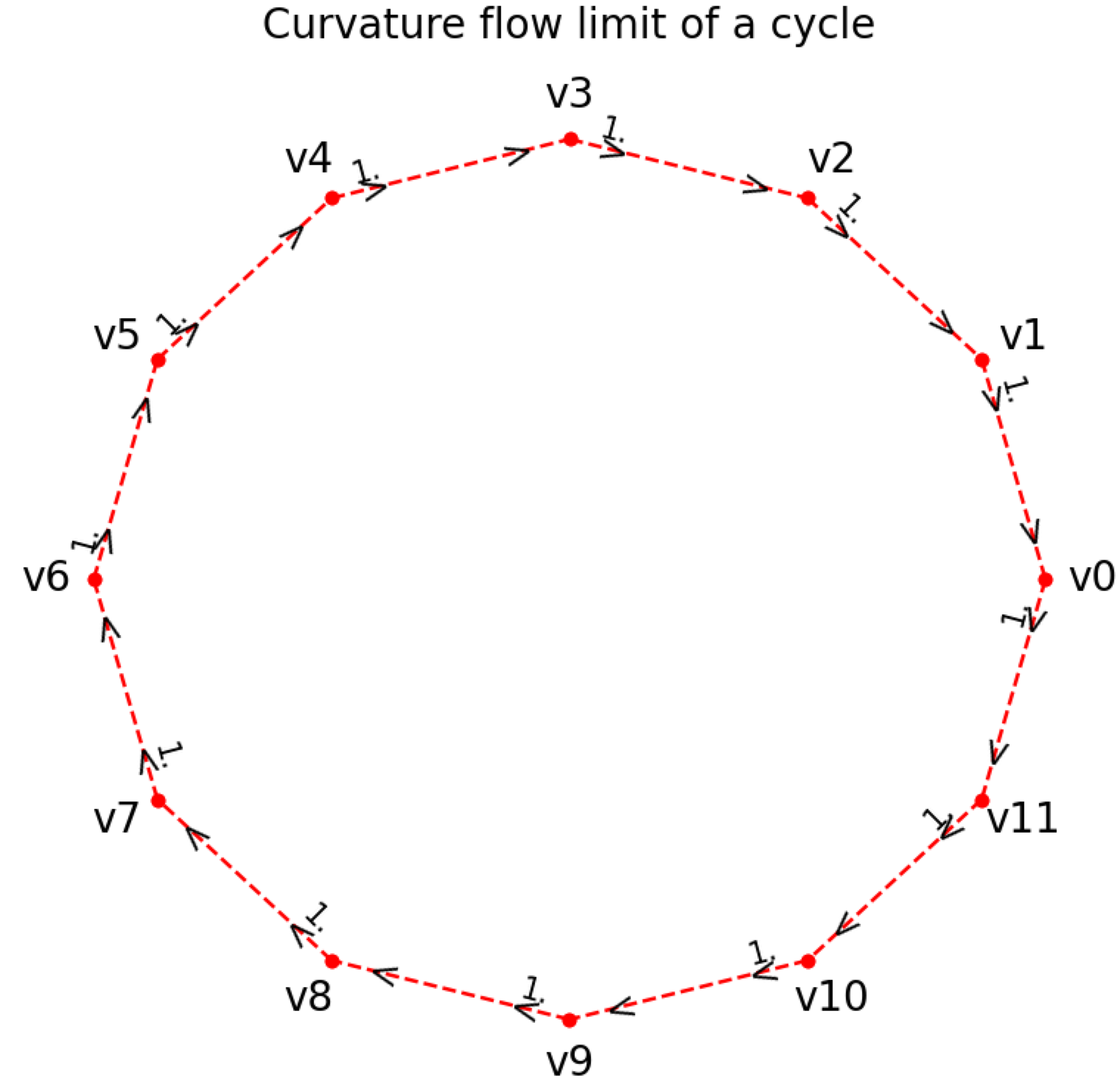}
\caption{Examples of numerical curvature flow limits of a path and of a cycle of length $12$.}
\label{fig:path-cycle-graph}
\end{figure}
\end{ex}

\FloatBarrier

Before we provide the following result about path limits, let us first introduce the notion of a two-sided degenerate edge of a weighted graphs $(G,P)$: An edge $\{x,y\} \in E$ of $G$ is called \emph{two-sided degenerate} if its transition rates vanish in both directions, that is, $p_{xy} = p_{yx} = 0$. Similarly, an edge is called \emph{numericall two-sided degenerate}, if its transition rates in both directions are $< \rm{threshold}$. Such edges are displayed in the routine {\texttt{display_weighted_graph}} as dotted black lines (see, e.g., the edges $\{v_2,v_9\}$ and $\{v_8,v_9\}$ on the right hand side of Figure \ref{fig:random-graph}).

\begin{prop}[Flow limits of paths of length $\ge 3$]
  Let $(G,P_0)$ be a weighted path of length $N \ge 3$ with consecutive vertices $v_0,\dots,v_{N-1}$. Let $P(t)$ be its corresponding curvature flow converging to a limit $P^\infty = \lim_{t \to \infty} P(t)$, such that $(G,P^\infty)$ does not have two-sided degenerate edges. Then this limit is neither totally degenerate nor is it not non-degenerate (that is, it contains both green and dashed red edges). Moreover, the components of the non-degenerate subgraph $G_W$ are paths of lengths $\le 2$ and they are separated from each other by at least two degenerate edges. If a component of $G_W$ is a path of length $1$, that is, just one edge $\{v_j,v_{j+1}\}$, then we have $p_{j,j+1}^\infty=p_{j+1,j}^\infty = 1$.  
\end{prop}

\begin{proof}
  By the last statement of Theorem \ref{thm:curvflowprops}, we only need to prove these statements for curvature sharp weighting schemes $P = (p_{ij})_{0\le i,j \le N}$ of paths of length $N \ge 3$. Such weighting schemes can never be non-degenerate by \cite[Prop. 1.11]{CKLMPS-22}. 
  
  For the proof that curvature sharp weighting schemes can never be totally degenerate, we note that we cannot have two consecutive degenerate edges $\{v_j,v_{j+1}\},\{v_{j+1},v_{j+2}\} \in E$ with $p_{j,j+1}=1=p_{j+2,j+1}$ (since this would imply $p_{j+1,j} = p_{j+1,j+2} = 0$, contradicting to $p_{j+1,j} + p_{j+1,j+2} = 1$). Therefore, since $p_{01}=1$, any totally degenerate weighting scheme would require $p_{j,j+1}=1$ for any $j \ge 0$, in particular, $p_{N-2,N-1}=1$, but this contradicts to the fact that we also have $p_{N-1,N-2}=1$ and that the edge $\{v_{N-1},v_{N-2}\}$ needs to be degenerate. 
  
  A curvature sharp weighting scheme cannot have more than two consecutive non-degenerate edges. This can be seen as follows: At any vertex $v_j$ with $p_{j,j-1},p_{j,j+1} > 0$ we must have $p_{j-1,j} = p_{j+1,j}$. This follows from the arguments in the proof of \cite[Lemma 4.1]{CKLMPS-22} (namely, since the vertex $v_j$ is not contained in any triangle, we have
  $p_{j-1,j} = p_{y+1,j} = p_{j,j-1}p_{j-1,j} + p_{j,j+1}p_{j+1,j}$.)
  If $0 < p_{j-1,j} = p_{j+1,j} < 1$, we could iterate this argument backward and forward and would end up with the fact that all entries of the matrix $P$ above the diagonal and below the diagonal would lie in $(0,1)$, which is a contradiction to $p_{01}=1$. So we must have $p_{j-1,j}=p_{j+1,j} \in \{0,1\}$. Therefore, we cannot have two consecutive indices $j \in \{0,\dots,N-1\}$ with $0 < p_{j-1,j} = p_{j+1,j} < 1$ which would exist in the case of three consecutive non-degenerate edges. So the components of $G_W$ are paths of length $\le 2$. 
  
  Moreover, any gap between two consecutive non-degenerate edges must be at least two edges: this follows from the fact that if a non-degenerate edge $\{v_j,v_{j+1}\}$ is followed by a degenerate edge $\{v_{j+1},v_{j+2}\}$, then we must have $p_{j+1,j+2}=0$: if we had $p_{j+1,j+2}>0$, then we had $p_{j+1,j},p_{j+1,j+2} > 0$ and, therefore $0 < p_{j+1,j} = p_{j+2,j+1}$ and $\{v_{j+1},v_{j+2}\}$ would be non-degenerate, which is a contradiction. Similarly, if a degenerate edge $\{v_k,v_{k+1}\}$ is followed by a non-degenerate edge $\{v_{k+1},v_{k+2}\}$, we must have $p_{k+1,k} = 0$. Combining both facts implies that there cannot be a single degenerate edge separating two components of $G_W$. These arguments show also that components of $G_W$ which are single edges $\{v_j,v_{j+1}\}$ must satisfy $p_{j,j+1}=p_{j+1,j}=1$ since the adjacent degenerate edges have one-sided transition rates pointing towards this component.
\end{proof}

Similar arguments as above can be used to prove for cycles that the components of any flow limit of a weighted cycle are again paths of length $\le 2$, unless the limit is non-degenerate. Such non-degenerate limits exist for cycles, namely, the simple random walks, but experiments show that simple random walks are very unstable stationary solutions of the curvature flow. Small perturbations of simple random walks do not converge back to the simple random walk (unless our cycle is $K_3$) but converge usually to a degenerate limit. 
Finally, if a totally degenerate cycle limit does not have two-sided degenerate edges, then its transition rates must all be either oriented clockwise or anti-clockwise for, otherwise, we would necessarily have a vertex with transition rates of both incident degenerate edges pointing towards this vertex. This would fail to satisfy the Markovian condition at this vertex.

\bigskip  

Paths and cycles of length $N \ge 3$ have the property that no edge is contained in a triangle. We like to finish section by a general statement about the curvature flow for edges not contained in triangles.

\begin{prop} \label{prop:zeronotriangle}
  Let $(G,P_0)$ be a weighted Markovian graph without laziness and $(P(t))_{t \ge 0}$ be its associated normalized curvature flow. If we have, for some $t_0 \ge 0$ and an edge $e = \{x,y\} \in E$, $p_{xy}(t_0)=0$ and $e$ is not contained in a triangle of $G$, then we have
  $$ p_{xy}(t) = 0 \quad \text{for all $t \ge t_0$}. $$
\end{prop}

\begin{proof}
  This proposition is an easy consequence of the flow equation \eqref{eq:flowdiffeq-explicit}. Since we assume no laziness, we have $p_{yy}(t)=0$, and since $e$ is not contained in a triangle, the last term of \eqref{eq:flowdiffeq-explicit}, denoted by $(*)$, is zero and the statement follows now from the uniqueness of the solution satisfying $p_{xy}(t_0)=0$.
\end{proof}

 \subsection{Complete graphs} 
 
The simple random walk on any complete graph is a non-degenerate curvature sharp Markovian weighting scheme. Our experiments show that any non-degenerate initial weighting scheme $P_0$ on $K_n$ converges to the simple random walk, that is $p_{jk}^\infty = \frac{1}{n-1}$. The following example shows that convergence to the simple random walk appears even if the initial weighting scheme is degenerate. This is not in contradiction ot Proposition \ref{prop:zeronotriangle} since, in a complete graph $K_n$, $n \ge 3$, every edge is contained in a triangle. Example \ref{ex:complete} is the only exception in this section where we allow an initial weighting scheme to have degenerate edges. 

\begin{ex}[A degenerate weighted complete graph with 6 vertices]
\label{ex:complete}
Let $(G,P_0)$ be the complete weighted Markovian graph with vertex set $V=\{v_0,\dots,v_5\}$ and
$$ P_0 = \begin{pmatrix} 0 & 0.2 & 0.1 & 0.2& 0 & 0.5 \\
0.1 & 0 & 0.3 & 0.25 & 0.25 & 0.1 \\
0.2 & 0 & 0 & 0.3 & 0.15 & 0.35 \\
0.3 & 0.5 & 0.1 & 0 & 0.1 & 0 \\
0.2 & 0.3 & 0.3 & 0.2 & 0 & 0 \\
0.6 & 0.1 & 0.1 & 0.2 & 0 & 0 \end{pmatrix}, $$
as illustrated in Figure \ref{fig:complete-graph}. Note that the edges $\{v_0,v_4\}, \{v_1,v_2\}, \{v_3,v_5\}$ and $\{v_4,v_5\}$ of this initial weighting scheme are degenerate, with the latter being two-sided degenerate. The numerical curvature flow has numerical convergence time $t_{\rm{max}} = 18.5$ (with respect to ${\rm{lim}}_{\rm{tolerance}} = 0.001$) with the simple random walk as its numerical flow limit. Figure \ref{fig:complete-graph-trans-rates} presents the transition rates of the vertices $v_2$ and $v_5$. Particular interesting are the functions $p_{2,1}(t)$ and $p_{5,4}(t)$ for $t \in [0,t_{\rm{max}}]$, since their initial values are zero. 

\begin{figure}[h]
\includegraphics[width=0.49\textwidth]{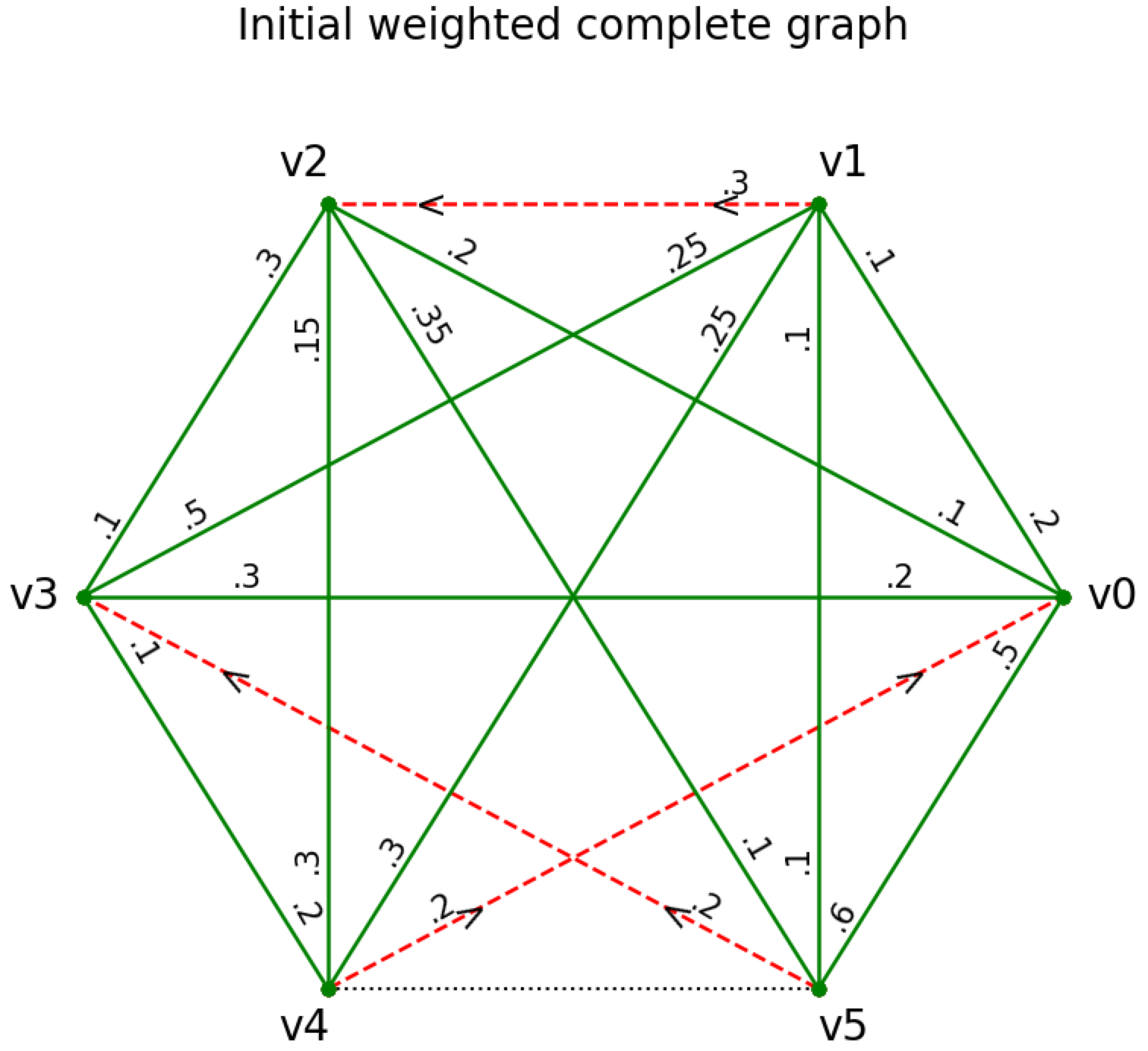}
\includegraphics[width=0.49\textwidth]{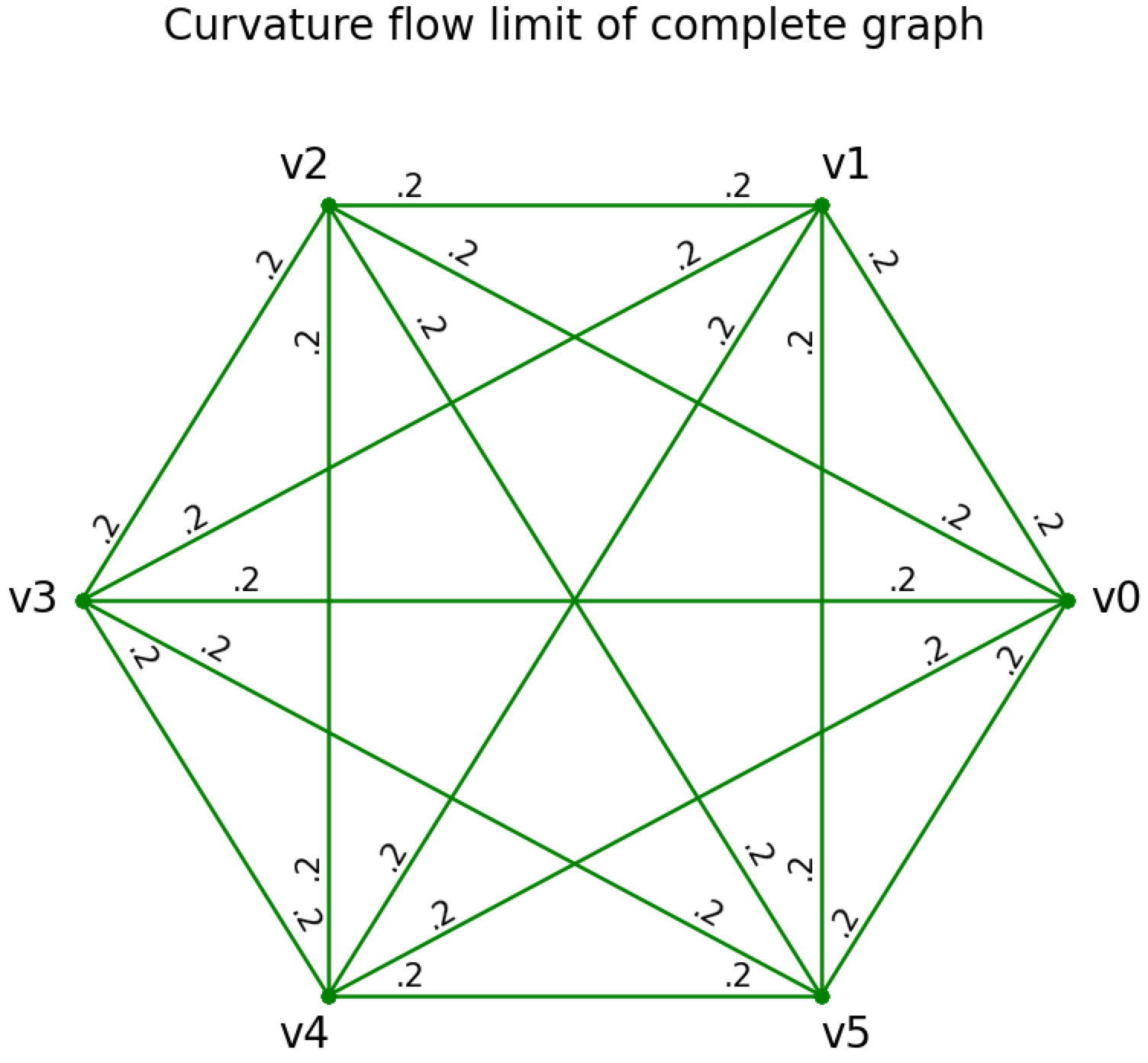}
\caption{Curvature flow of a complete graph with $6$ vertices with degenerate initial weighting scheme $P_0$ (left hand side) and numerical flow limit $P(t_{\rm{max}})$, the simple random walk (right hand side)}
\label{fig:complete-graph}
\end{figure}

\FloatBarrier

\begin{figure}[h]
\includegraphics[width=0.49\textwidth]{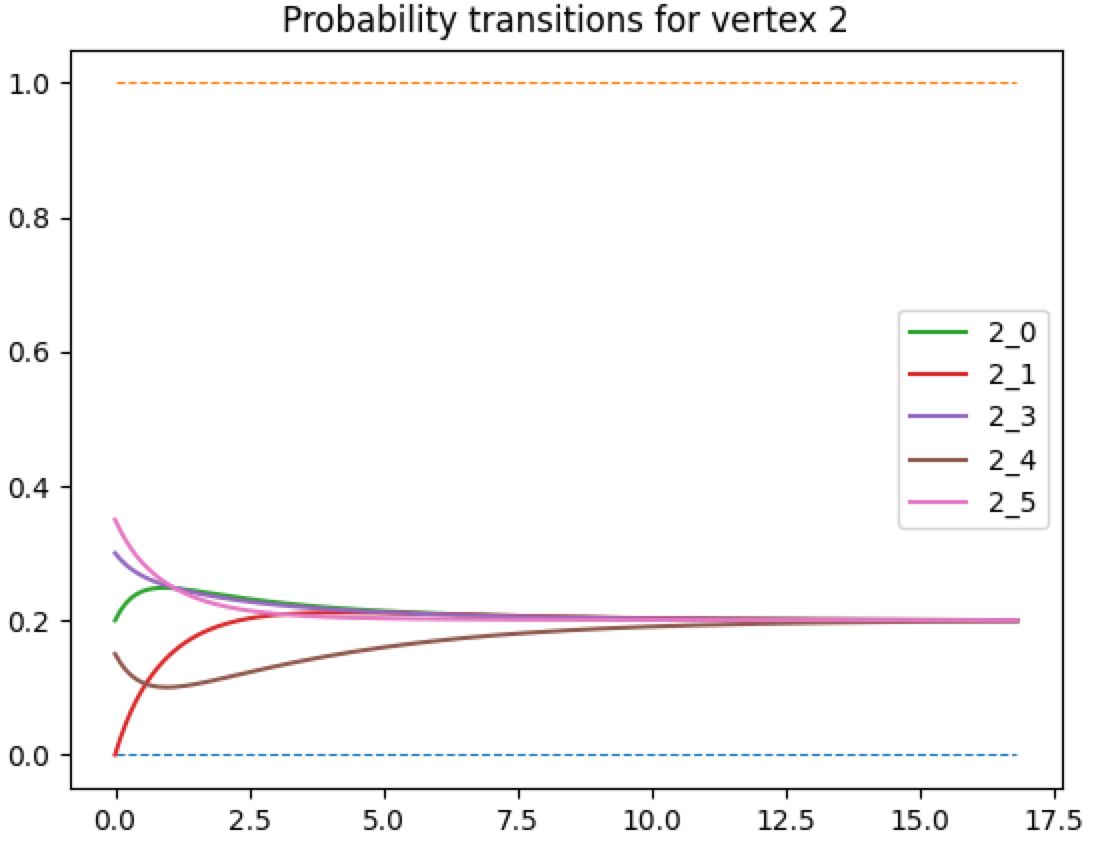}
\includegraphics[width=0.49\textwidth]{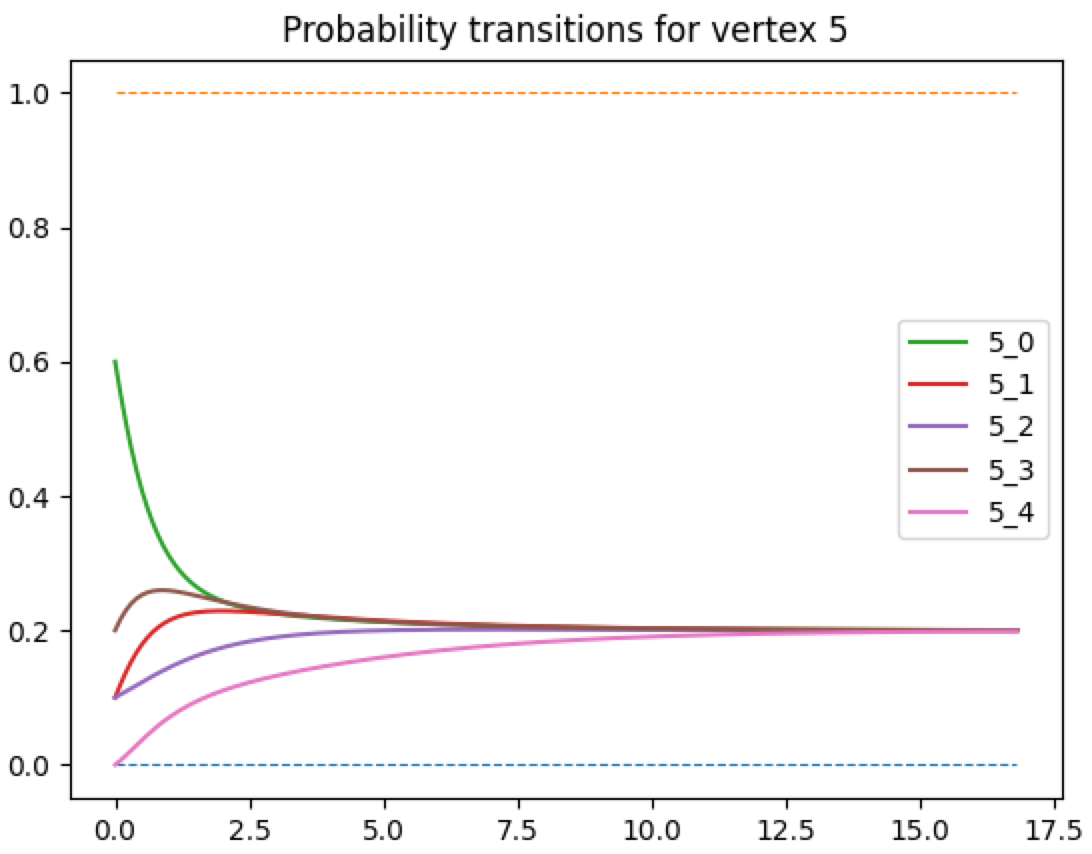}
\caption{Transition rates of vertices $v_2$ and $v_5$ of a complete graph with $6$ vertices under the curvature flow}
\label{fig:complete-graph-trans-rates}
\end{figure}
\end{ex}
	
	
Based on our experiments, we conjecture the following, which is a strengthening of \cite[Conjecture 1.8]{CKLMPS-22}.

\begin{conj}[Curvature flow of complete graphs] \label{conj:complgraph}
Let $P_0$ be a Markovian weighting scheme without laziness on a complete graph $K_n=(V,E)$ with $n \ge 2$ such that, for every proper subset $W \subset V$, there exist $x,x' \in W$ and $y,y' \in \in V \setminus W$ with $p_{xy} > 0$ and $p_{y'x'} > 0$. Then the curvature flow has a limit $P^\infty$ which is the simple random walk.  
\end{conj}

\FloatBarrier

\subsection{Wedge sums of complete graphs}
Let $G_1=(V_1,E_1)$ and $G_2=(V_2,E_2)$ be two combinatorial graphs and $x_1 \in V_1$ and $x_2 \in V_2$. By merging the
vertices $x_1$ and $x_2$ into one new vertex $x$ which inherits
the incident edges of both vertices $x_1$ and $x_2$, we obtain a new combinatorial graph, which we denote the wedge sum of $G_1$ and $G_2$. In this subsection, we consider flow limits of wedge sums of complete graphs. 

\begin{ex}[A wedge sum of a $K_4$, $K_5$, $K_2$ and $K_3$] We consider the curvature flow on the wedge sum $G=(V,E)$ of complete graphs presented in Figure \ref{fig:wedge-sum} with random initial weighting schemes $P = P_0$. The adjacency matrix $A$ of this graph is generated with our code in the following way:
\begin{lstlisting}[language=Python]{}
A1 = wedge_sum(complete(4), complete(5), 2, 1)
A2 = wedge_sum(A1, complete(2), 6, 0)
A = wedge_sum(A2, complete(3), 8, 0)
\end{lstlisting}

\begin{figure}[h]
\includegraphics[width=.5\textwidth]{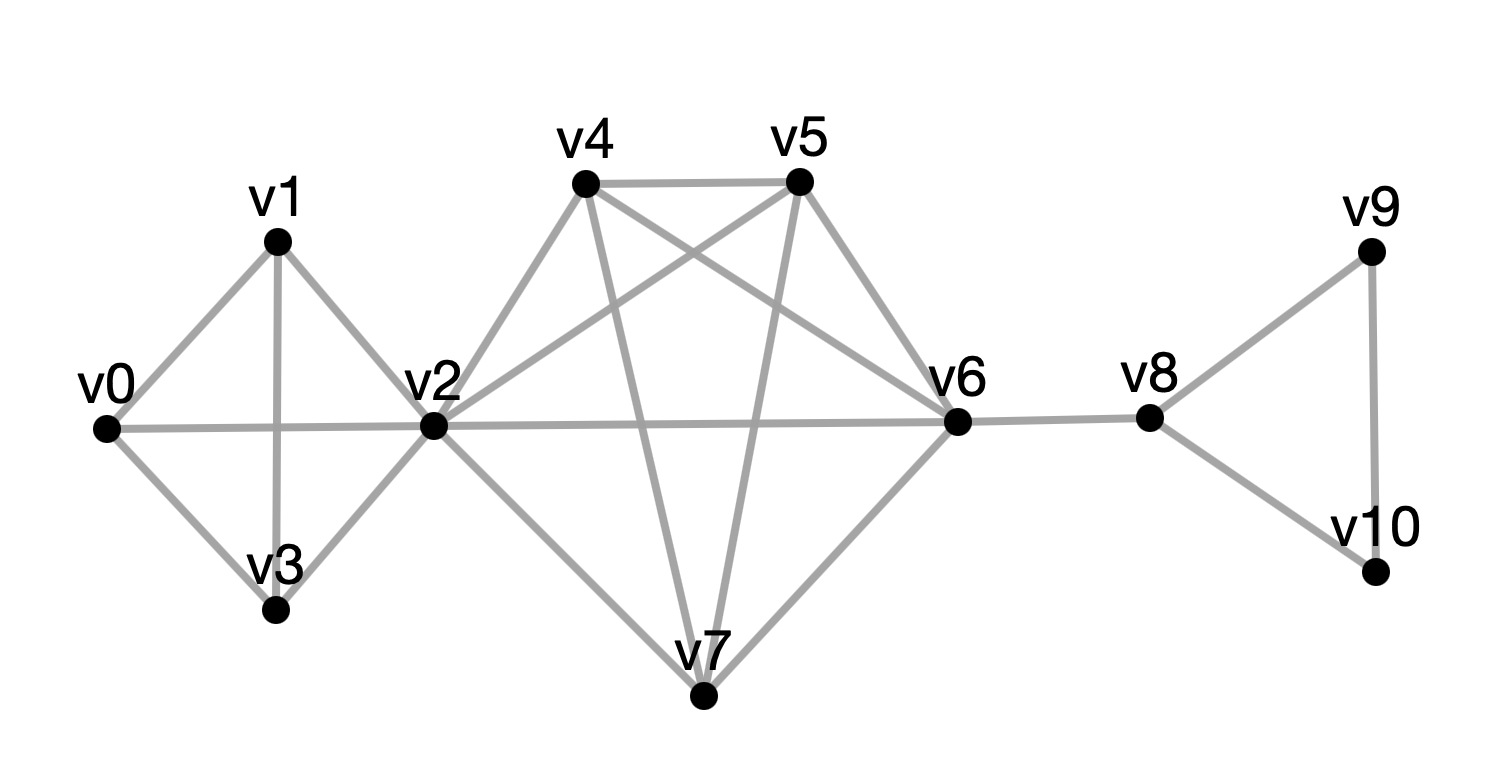}
\caption{A wedge sum of a $K_4$, $K_5$, $K_2$ and a $K_3$}
\label{fig:wedge-sum}
\end{figure}

\begin{figure}
\begin{minipage}[h]{0.49\textwidth}
{\scriptsize{\begin{center}
\begin{multline*} 
P_0 = \\ \begin{pmatrix} 
0 & 0.4 & 0.5 & 0.1 & 0 & 0 & 0 & 0 & 0 & 0 & 0 \\
0.1 & 0 & 0.4 & 0.5 & 0 & 0 & 0 & 0 & 0 & 0 & 0 \\
0.1 & 0.1 & 0 & 0.3 & 0.1 & 0.1 & 0.1 & 0.2 & 0 & 0 & 0 \\
0.2 & 0.5 & 0.3 & 0 & 0 & 0 & 0 & 0 & 0 & 0 & 0 \\
0 & 0 & 0.1 & 0 & 0 & 0.1 & 0.3 & 0.5 & 0 & 0 & 0 \\
0 & 0 & 0.1 & 0 & 0.1 & 0 & 0.4 & 0.4 & 0 & 0 & 0 \\
0 & 0 & 0.2 & 0 & 0.3 & 0.1 & 0 & 0.1 & 0.3 & 0 & 0 \\
0 & 0 & 0.5 & 0 & 0.1 & 0.3 & 0.1 & 0 & 0 & 0 & 0 \\
0 & 0 & 0 & 0 & 0 & 0 & 0.1 & 0 & 0 & 0.8 & 0.1 \\
0 & 0 & 0 & 0 & 0 & 0 & 0 & 0 & 0.5 & 0 & 0.5 \\
0 & 0 & 0 & 0 & 0 & 0 & 0 & 0 & 0.7 & 0.3 & 0 \end{pmatrix}
\end{multline*}
\end{center}}}
\medskip
\end{minipage}
\begin{minipage}[h]{0.49\textwidth}
\begin{center}
\includegraphics[width=\textwidth]{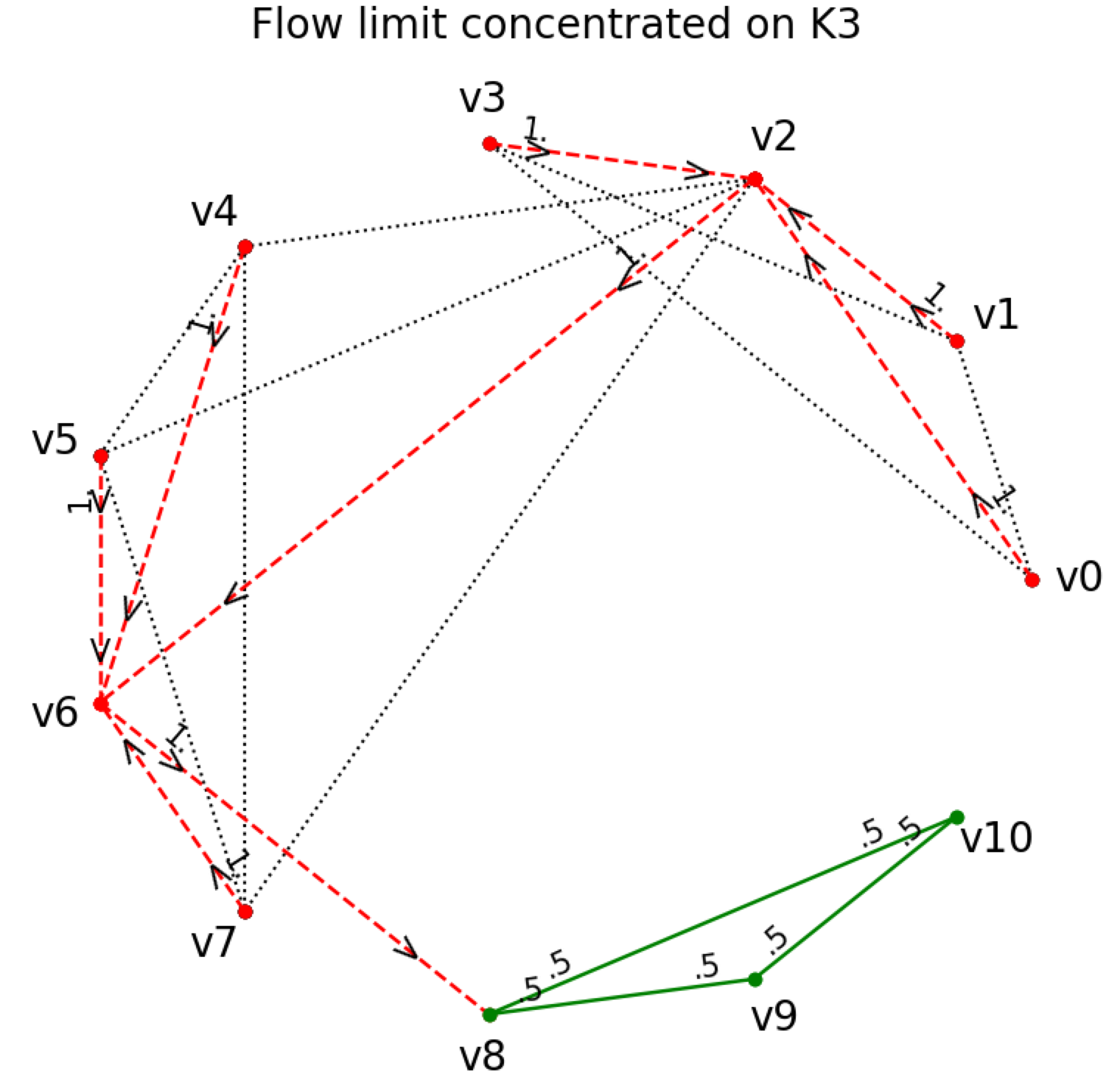}
\end{center}
\medskip
\end{minipage}
\caption{Initial weighting scheme (left) and numerical flow limit a simple random walk on $K_3$ (right)}
\label{fig:k3-winner}
\end{figure}

\begin{figure}
\begin{minipage}[h]{0.49\textwidth}
{\scriptsize{\begin{center}
\begin{multline*} 
P_0 = \\ \begin{pmatrix} 
0 & 0.3 & 0.3 & 0.4 & 0 & 0 & 0 & 0 & 0 & 0 & 0 \\
0.1 & 0 & 0.4 & 0.5 & 0 & 0 & 0 & 0 & 0 & 0 & 0 \\
0.3 & 0.2 & 0 & 0.1 & 0.1 & 0.1 & 0.1 & 0.1 & 0 & 0 & 0 \\
0.4 & 0.4 & 0.2 & 0 & 0 & 0 & 0 & 0 & 0 & 0 & 0 \\
0 & 0 & 0.3 & 0 & 0 & 0.1 & 0.1 & 0.5 & 0 & 0 & 0 \\
0 & 0 & 0.2 & 0 & 0.2 & 0 & 0.5 & 0.1 & 0 & 0 & 0 \\
0 & 0 & 0.4 & 0 & 0.1 & 0.1 & 0 & 0.3 & 0.1 & 0 & 0 \\
0 & 0 & 0.7 & 0 & 0.1 & 0.1 & 0.1 & 0 & 0 & 0 & 0 \\
0 & 0 & 0 & 0 & 0 & 0 & 0.5 & 0 & 0 & 0.4 & 0.1 \\
0 & 0 & 0 & 0 & 0 & 0 & 0 & 0 & 0.4 & 0 & 0.6 \\
0 & 0 & 0 & 0 & 0 & 0 & 0 & 0 & 0.1 & 0.9 & 0 \end{pmatrix}
\end{multline*}
\end{center}}}
\medskip
\end{minipage}
\begin{minipage}[h]{0.49\textwidth}
\begin{center}
\includegraphics[width=\textwidth]{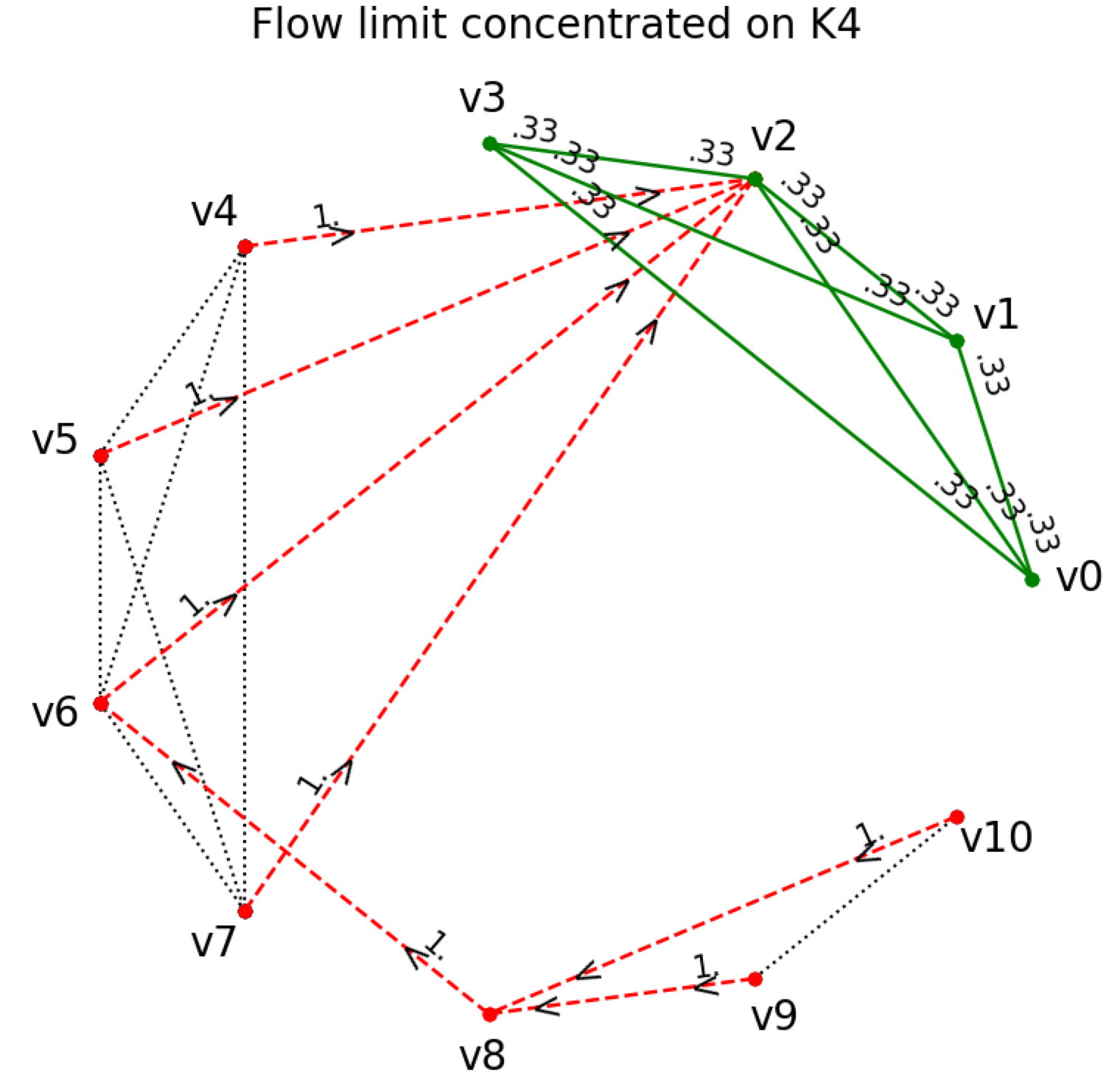}
\end{center}
\medskip
\end{minipage}
\caption{Initial weighting scheme (left) and numerical flow limit a simple random walk on $K_4$ (right)}
\label{fig:k4-winner}
\end{figure}

\FloatBarrier

Our experiments show that, depending on the initial weighting scheme $P_0$, the curvature flow converges to a limit which is concentrated in one of the complete graphs. More precisely, the limit weighting scheme $P^\infty$ represents a simple random walk on one of the complete graphs while, from all other vertices, there is a directed path of $\{0,1\}$ transition rates towards this particular complete graph. Figures \ref{fig:k3-winner}, \ref{fig:k4-winner} and \ref{fig:k5-winner} show the numerical curvature flow limits of various random initial weighting schemes. We carried out several runs of $100000$ numerical curvature flows with random initial weighting schemes to describe the limit behaviour of these flows quantitatively. The results are presented in Table \ref{tab:k4k5k2k3-statistics}. While more than $80\%$ of limits concentrate on the largest clique $K_5$, it is somewhat surprising that more limits concentrate on $K_3$ than on the larger subgraph $K_4$. Not a single flow limit ended up concentrating on $K_2$. The mean numerical convergence time is shortest for $K_3$, followed by $K_4$ and $K_5$ (with respect to $\lim_{\rm{tolerance}} = 0.001$). While most convergence times are below 100, there were maximal numerical convergence times well above 500.

\begin{figure}
\begin{minipage}[h]{0.49\textwidth}
{\scriptsize{\begin{center}
\begin{multline*} 
P_0 = \\ \begin{pmatrix} 
0 & 0.2 & 0.2 & 0.6 & 0 & 0 & 0 & 0 & 0 & 0 & 0 \\
0.2 & 0 & 0.3 & 0.5 & 0 & 0 & 0 & 0 & 0 & 0 & 0 \\
0.1 & 0.1 & 0 & 0.2 & 0.1 & 0.1 & 0.2 & 0.2 & 0 & 0 & 0 \\
0.3 & 0.3 & 0.4 & 0 & 0 & 0 & 0 & 0 & 0 & 0 & 0 \\
0 & 0 & 0.1 & 0 & 0 & 0.2 & 0.3 & 0.4 & 0 & 0 & 0 \\
0 & 0 & 0.2 & 0 & 0.4 & 0 & 0.2 & 0.2 & 0 & 0 & 0 \\
0 & 0 & 0.3 & 0 & 0.2 & 0.1 & 0 & 0.2 & 0.2 & 0 & 0 \\
0 & 0 & 0.5 & 0 & 0.1 & 0.3 & 0.1 & 0 & 0 & 0 & 0 \\
0 & 0 & 0 & 0 & 0 & 0 & 0.6 & 0 & 0 & 0.2 & 0.2 \\
0 & 0 & 0 & 0 & 0 & 0 & 0 & 0 & 0.5 & 0 & 0.5 \\
0 & 0 & 0 & 0 & 0 & 0 & 0 & 0 & 0.3 & 0.7 & 0 \end{pmatrix}
\end{multline*}
\end{center}}}
\medskip
\end{minipage}
\begin{minipage}[h]{0.49\textwidth}
\begin{center}
\includegraphics[width=\textwidth]{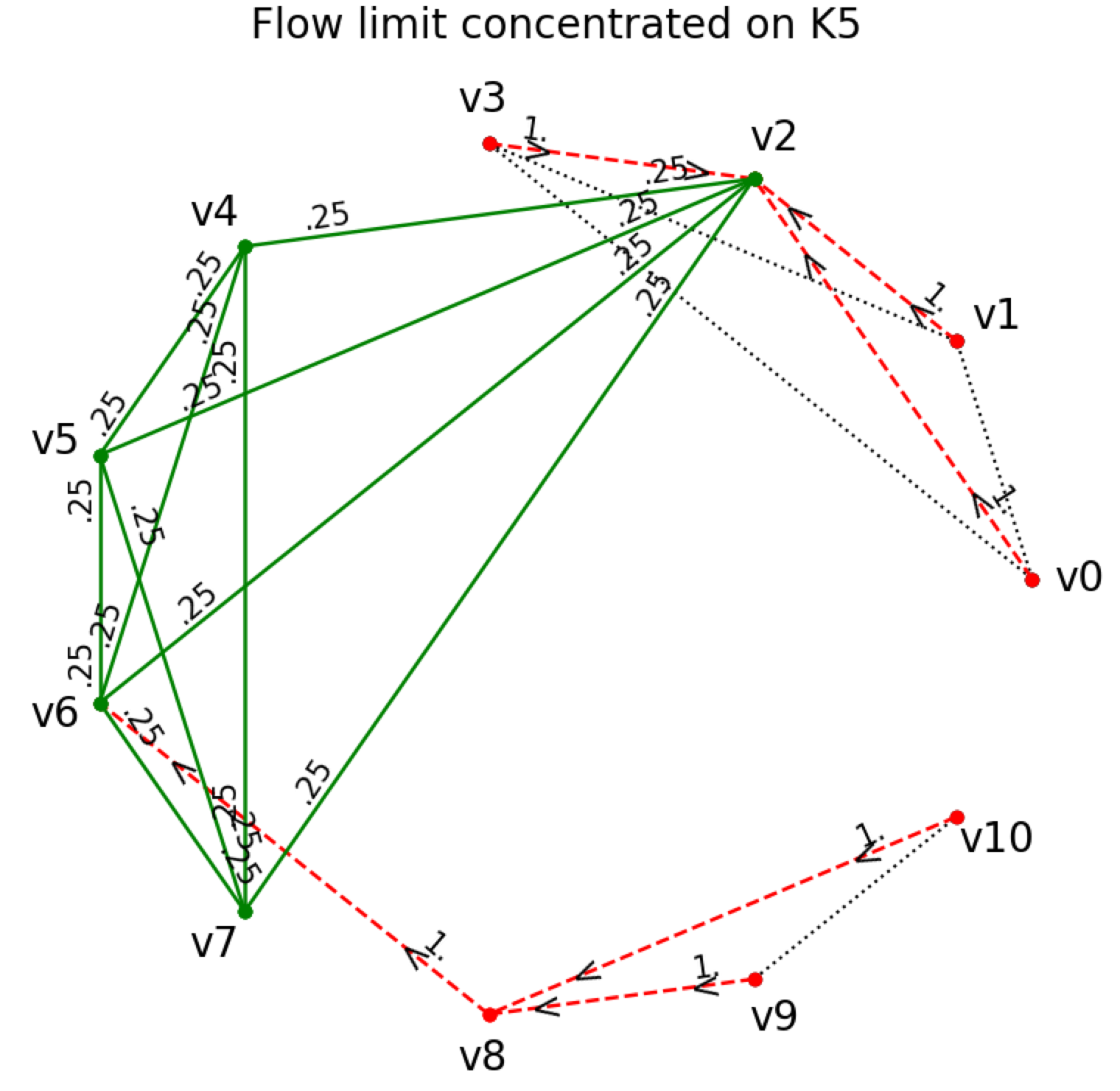}
\end{center}
\medskip
\end{minipage}
\caption{Initial weighting scheme (left) and numerical flow limit a simple random walk on $K_5$ (right)}
\label{fig:k5-winner}
\end{figure}
\FloatBarrier
\begin{table}[h]
    \centering
    \begin{tabular}{l|llll}
         & $K_4$ & $K_5$ & $K_2$ & $K_3$ \\ \hline\\[-.3cm]
    concentration of flow limits & $8.7\%$ & $80.5\%$ & $0\%$ & $10.8\%$ \\
    mean convergence time & $18.9$ & $24.7$ & -- & $17.7$ \\[.2cm]
    \end{tabular}
    \caption{Statistics of flow limit concentrations on the complete subgraphs $K_4,K_5,K_2$ and $K_3$ of their wedge sum, together with mean convergence times}
    \label{tab:k4k5k2k3-statistics}
\end{table}
\end{ex}

The limit behaviour described in the above example seems to be common for many wedge sums of complete graphs. It is, however, not always true that flow limits concentrate on just one of the constituents of a wedge sum. A path of length $\ge 2$ can be viewed as a wedge sum of consecutive $K_2$'s, and we have seen in Subsection \ref{subsec:paths-cyles} that flow limits will concentrate on more than only one of these $K_2$'s (see left hand side of Figure \ref{fig:path-cycle-graph}). Another special case of a wedge sum is a dumbbell which is our next example.

\begin{ex}[A symmetric weighted dumbbell] The weighted graph $(G,P_0)$ in this example is a wedge sum of a $K_5$, $K_2$ and another $K_5$, together with a simple random walk as initial weighting scheme (see line 4 in the code). This situation can be set up by the following code:
\begin{lstlisting}[language=Python]{}
A1 = complete(5)
A2 = complete(5)
A = bridge_at(A1,A2,0,0)
P_0 = srw(A)
\end{lstlisting}
The numerical convergence time is $79$ and the limit of the numerical curvature flow concentrates on the ``bridge'' $K_2$ between the two $K_5$'s, as illustrated in Figure \ref{fig:dumbbell-graph}. (To obtain the flow limit illustrated at the right hand side of this figure, users should choose $\lim_{\rm{tolerance}} = 0.0001$.)

\begin{figure}[h]
\includegraphics[width=0.49\textwidth]{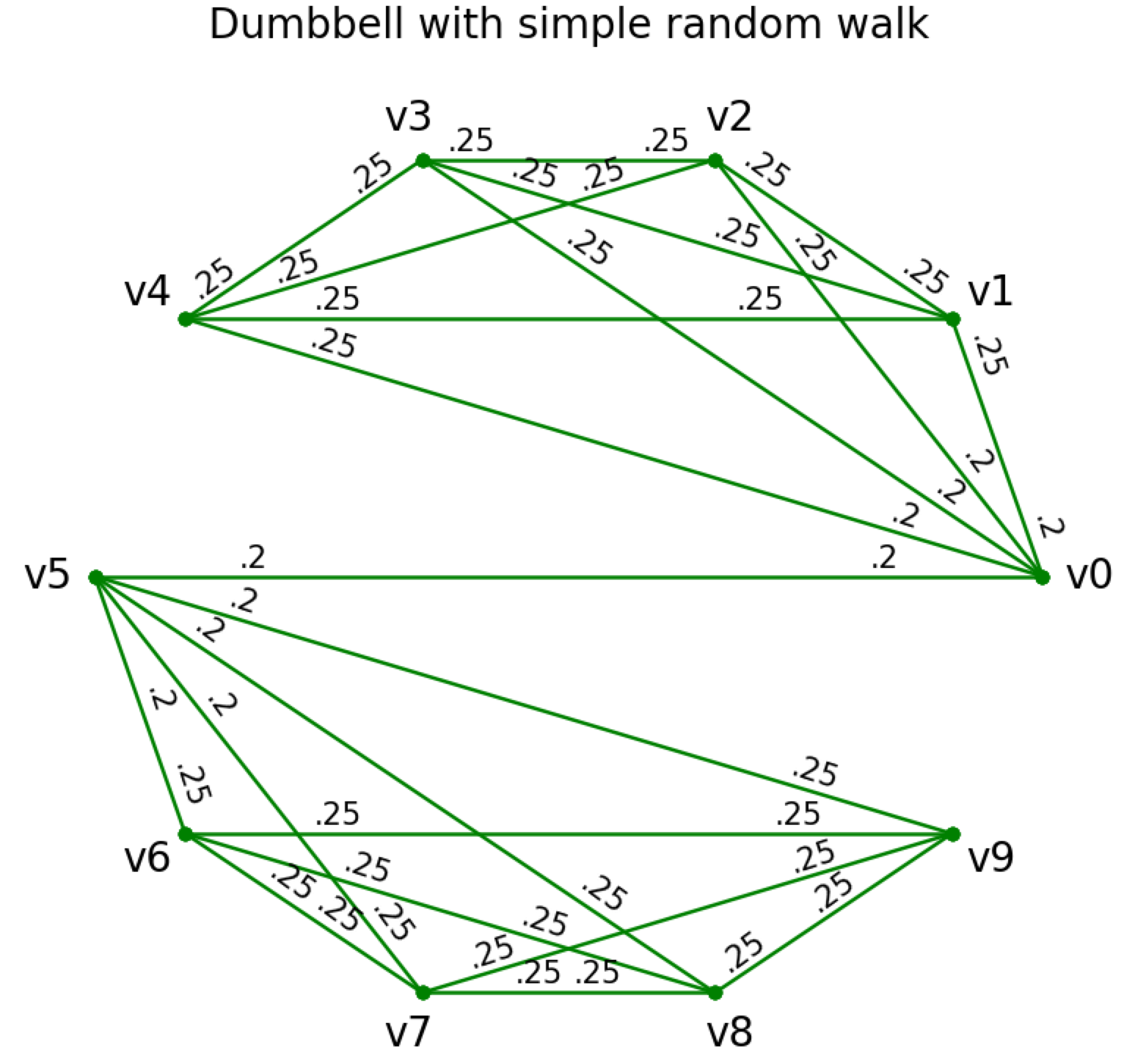}
\includegraphics[width=0.49\textwidth]{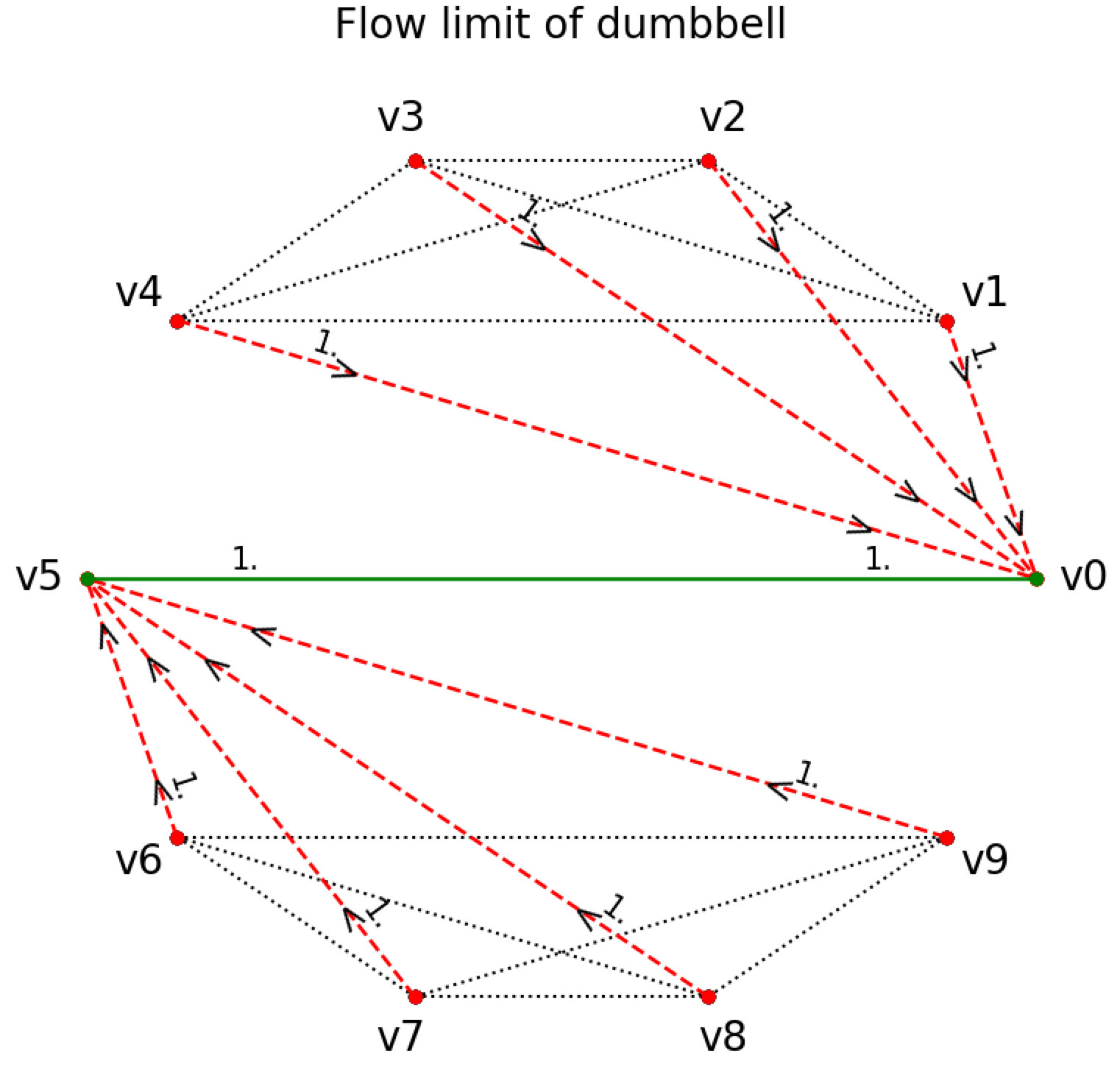}
\caption{Curvature flow of a dumbbell as a wedge sum of a $K_5$, $K_2$ and another $K_5$ with simple random walk initial weighting scheme $P_0$ (left hand side) and numerical flow limit concentrated on $K_2$ (right hand side)}
\label{fig:dumbbell-graph}
\end{figure}

\FloatBarrier

This limit could have been predicted assuming that the initial weighted graph symmetry across the ``bridge'' is preserved under the curvature flow and that the limit concentrates on only one of the complete graphs $K_5, K_2, K_5$. If instead of the simple random walk, a random initial weighting scheme would have been chosen on the dumbbell $G$, the limit would have usually concentrated on one of the two $K_5$'s.
\end{ex}

\subsection{Cartesian products of complete graphs}\label{sec:cartprodcompl}

We have the following result for Cartesian products of complete graphs:

\begin{thm}
  Let $G$ be the Cartesian product of two complete graphs $K_{n+1}$ and $K_{m+1}$ with the non-lazy simple random walk $P$ as initial weighting scheme. Then the curvature flow $P(t)$ converges to a limit as $t \to \infty$ with limit transition rates 
  \begin{eqnarray*}
    a^\infty &=& \frac{m+3}{2nm+3n+3m}, \\
    b^\infty &=& \frac{n+3}{2nm+3n+3m},
  \end{eqnarray*}
  where $a^\infty$ are the transition rates along edges between $(x,x')$ and $(y,x')$ with $x,y \in K_{n+1}$, $x \sim y$ and $x' \in K_{m+1}$ (``horizontal edges''), and $b^\infty$ are the transition rates along edges between $(x,x')$ and $(x,y')$ with $x \in K_{n+1}$ and $x',y' \in K_{m+1}$, $x' \sim y'$ (``vertical edges'').
\end{thm}

\begin{proof}
By symmetry of the configuration, we have only two types $a(t)$ and $b(t)$ of transition rates for the curvature flow at time $t \ge 0$ (for horizontal and vertical edges, respectively) with
$$ a(0) = b(t) = \frac{1}{n+m} $$
and
$$ na(t) + mb(t) = 1, $$
due to the Markovian property. This implies that $b(t) = \frac{1-na(t)}{m}$ and $b'(t) = - \frac{-na'(t)}{m}$, and we only need to consider an ordinary differential equation for $a$ with initial condition $a(0) = \frac{1}{n+m}$. We can also assume without loss of generality that $n \le m$. We derive from the explicit description \eqref{eq:flowdiffeq-explicit} of the curvature flow that
\begin{multline} \label{eq:diffeqprodknkm} 
a'(t) = \\ a(t) \left( -4a(t) - 2(n-1)a(t) + 4 (na^2(t)+mb^2(t)) + n(n-1)a^2(t) + m(m-1)b^2(t) \right) \\
+ (n-1)a^2(t) = \\
n(n+3)a^3(t) + m(m+3)a(t) \left( \frac{1-na(t)}{m} \right)^2 -(n+3)a^2(t) = \\ 
a(t)\left( n\left( 2n+3+\frac{3n}{m}\right)a^2(t) - \left( 3n+3+\frac{6n}{m} \right)a(t) + 1 + \frac{3}{m} 
\right).
\end{multline}
If this differential equation converges as $t \to \infty$, its limit $a^\infty$ must satisfy
$$ a^\infty \in \left\{ 0,\frac{1}{n},\frac{m+3}{2nm+3n+3m} \right\}, $$
that is
$$ (a^\infty,b^\infty) \in \left\{ \left(0,\frac{1}{m}\right),\left(\frac{1}{n},0\right),\left(\frac{m+3}{2nm+3n+3m},\frac{n+3}{2nm+3n+3m}\right) \right\}, $$
with 
$$ b^\infty = \lim_{t \to \infty} b(t) = \lim_{t \to \infty} \frac{1-na(t)}{m} = \frac{1-na^\infty}{m}. $$
Our assumption $n \le m$ implies that we have
$$ \frac{1}{n+m} \le \frac{m+3}{2nm+3n+3m} < \frac{1}{n} $$
and that the right hand side of \eqref{eq:diffeqprodknkm} is strictly positive for $(at)$ in the interval
$$ \left[ \frac{1}{n+m},\frac{m+3}{2nm+3n+3m} \right), $$
zero at $a(t) = \frac{m+3}{2nm+3n+3m}$,
and strictly negative for $a(t)$ on the interval 
$$ \left( \frac{m+3}{2nm+3n+3m},\frac{1}{n} \right). $$
These monotonicity properties force the function $a(t)$ to converge to the limit $a^\infty = \frac{m+3}{2nm+3n+3m}$.
\end{proof}

\begin{ex}[Flow limit of $K_3 \times K_4$ with simple random walk] The following code computes the numerical flow limit for $K_3 \times K_4$ with the simple random walk as initial weigthing scheme.
\begin{lstlisting}[language=Python]{}
n,m = 2,3
A = cart_prod(complete(n+1),complete(m+1))
P = srw(A)
limit = norm_curv_flow_lim(A, P)[0]
display_weighted_graph(A, P, "Initial weighting scheme of K3 x K4")
display_weighted_graph(A, limit, title="Curvature flow limit")
\end{lstlisting}
\end{ex}
The initial transition rates are all equal to $1/5=0.2$, and the transition rates of the numerical limit are $0.22\dots \approx 2/9$ and $0.19\dots \approx 5/27$, as predicted by the above theorem.

\subsection{Totally degenerate flow limits}
 
As mentioned before, many curvature limits are totally degenerate. Therefore, it is worth to investigate properties of those particular limits. 


\begin{ex}[The octahedron with a totally degenerate flow limit] \label{ex:octahedron} Let $G=(V,E)$ with $V = \{v_0,\dots,v_5\}$ be the unmixed graph representing the octahedron, as illustrated in Figure \ref{fig:octahedron} (left hand side). While the simple random walk without laziness is a stationary solution of the normalised curvature flow, any small perturbation of this initial weighting scheme leads to another curvature sharp limit, which is totally degenerate. For example, the initial weighting scheme
\begin{equation} \label{eq:P0-octahedron} 
P_0 = \begin{pmatrix} 
0&0.26&0&0.24&0.25&0.25\\
0.25&0&0.25&0&0.25&0.25\\
0&0.25&0&0.25&0.25&0.25\\
0.25&0&0.25&0&0.25&0.25\\
0.25&0.25&0.25&0.25&0&0\\
0.25&0.25&0.25&0.25&0&0
\end{pmatrix}
\end{equation}
converges to the totally degenerate limit illustrated in Figure \ref{fig:octahedron} (right hand side) under the numerical curvature flow.

\begin{figure}[h]
\includegraphics[width=0.49\textwidth]{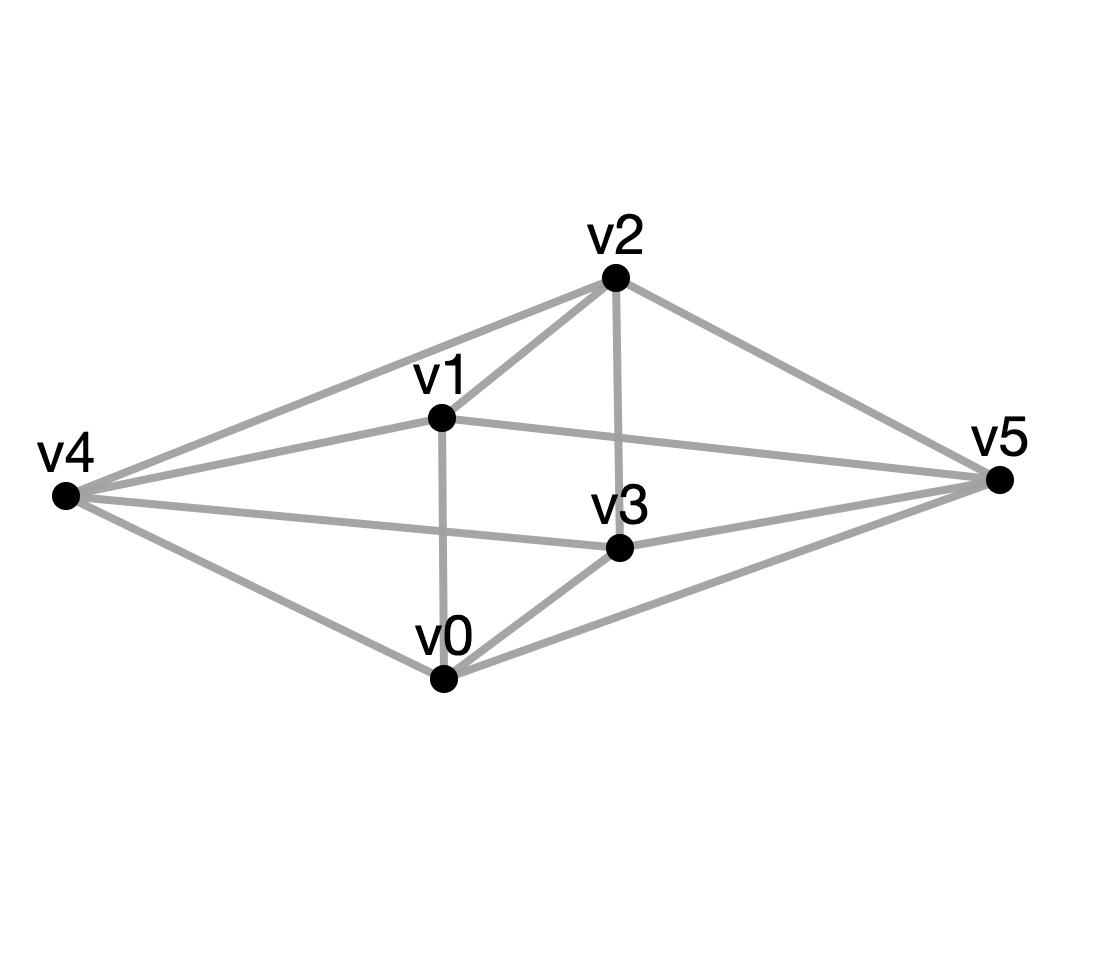}
\includegraphics[width=0.49\textwidth]{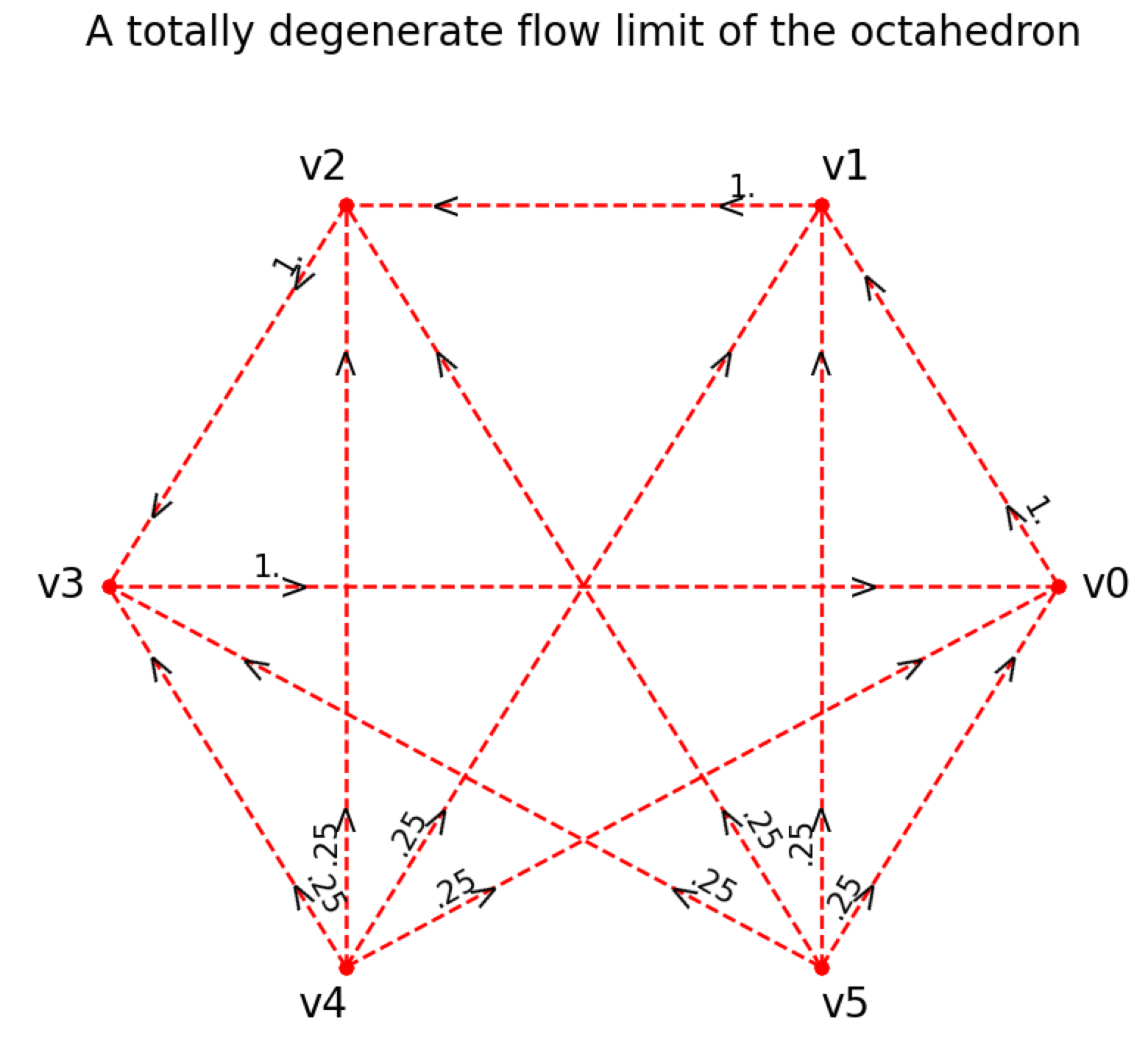}
\caption{Left hand side: An octahedron with vertices $v_0,\dots,v_5$ . Right hand side: A numerical flow limit of the octahedron with a small perturbation of simple random walk as initial weighing scheme}
\label{fig:octahedron}
\end{figure}
\end{ex}

\FloatBarrier


\FloatBarrier

The following considerations show that the limit in Example \ref{ex:octahedron} is essentially the only totally degenerate curvature sharp weighting scheme without two-sided degenerate edges for the octahedron
(see Proposition \ref{prop:octahedron} below). We start with a general unmixed combinatorial graph $G=(V,E)$ without isolated vertices. A totally degenerate weighting scheme $P$ assigns to each edge $\{x,y\} \in E$ either a direction ($x \to y$ if $p_{xy} >0$ and $y \to x$ if $p_{yx} >0$, illustrated by a red dashed line with an arrow) or the edge is two-sided degenerate (that is $p_{xy} = p_{yx}=0$, illustrated by a black dotted line). A first observation is that none of the vertices $x \in V$ can be a sink: the Markovian property requires that at least one edge incident to $x$ must be outward directed. The following lemma presents a useful property of triangles.

\begin{lemma} \label{lem:totdegencurvsharp}
  Let $(G,P)$ be a totally degenerate curvature sharp Markovian weighted graph without laziness. Then the directions of a triangle $T = \{x,y,z\} \subset V$ without two-sided degenerate edges cannot be oriented, that is, its edges cannot have the orientations $x \to y \to z \to x$ or $x \to z \to y \to x$.    
\end{lemma}

\begin{proof}
  Assume that a totally degenerate curvature sharp Markovian weigthing scheme without laziness contains a triangle $T = \{x,y,z\}$ with $p_{xz} , p_{zy}, p_{yx} > 0$, that is, we have an orientation $x \to z \to y \to x$. This means, in particular, that $p_{xy}=0$. We can read off the flow equation \eqref{eq:flowdiffeq-explicit} that curvature sharpness of a totally degenerate Markovian weighting scheme means
  \begin{equation} \label{eq:curvsharptotdeg} 
  p_{xy}(-2\sum_{y'\neq y} p_{yy'} + \sum_{y',y''} p_{xy'}p_{y'y''}) + \sum_{y'\neq y} p_{xy'}p_{y'y} = 0. 
  \end{equation}
  Since we have $p_{xy} = 0$, this means
  $$ 0 \le p_{xz} p_{zy} \le \sum_{y' \neq y} p_{xy'}p_{y'y} = 0, $$
  in contradiction to $p_{xz}, p_{zy} > 0$. The orientation $x \to y \to z \to x$ can be ruled out similarly.
\end{proof}




The main tool in the proof of the following proposition can be found in \cite[Exercise 25.14]{Pak-10}:
Assume a tessellation of the $2$-dimensional sphere carries an orientation along all its edges. For every vertex $v$ of the tessellation let ${\rm{ind}}(v) = 1-c(v)/2$, where $c(v)$ is the number of changes in the orientation of edges adjacent to $v$ (in cyclic order). For a face $f$, let ${\rm{ind}}(f) = 1-c(f)/2$, where $c(f)$ is the number of changes in the orientation (clockwise vs. anti-clockwise) of edges of $f$. Then we have
$$ \sum_v {\rm{ind}}(v) + \sum_{f} {\rm{ind}}(f) = 2. $$





\begin{prop} \label{prop:octahedron}
  Let $G=(V,E)$ be the octahedron, as illustrated in Figure \ref{fig:octahedron} (left hand side). Then $G$ has essentially only one totally degenerate non-lazy curvature sharp Markovian weighting scheme without two-sided
  degenerate edges, namely, we have (up to a permutation of the vertices corresponding to a graph automorphism)
  $$ p_{v_0,v_1} = p_{v_1,v_2} = p_{v_2,v_3} = p_{v_3,v_0} = 1 $$
  and 
  $$ p_{v_4,v_0} = p_{v_4,v_1} = p_{v_4,v_2} = p_{v_4,v_3} = 1/4, \quad p_{v_5,v_0} = p_{v_5,v_1} = p_{v_5,v_2} = p_{v_5,v_3} = 1/4. $$
\end{prop}

\begin{proof}
  We can think of the octahedron as a tessellation of the sphere by $8$ triangles, all of them not oriented, by Lemma \ref{lem:totdegencurvsharp}. This means that ${\rm{ind}}(f)=0$ for all faces of the octahedron. Therefore, we must have
  $$ \sum_v {\rm{ind}}(v) = 2, $$
  that is, at least two vertices of the octahedron must have ${\rm{ind}}(v)=1$, which means that each of them must be a source or a sink. The Markovian property rules out sinks, and two sources cannot be adjacent (otherwise they would be connected by a non-degenerate edge). So the octahedron must have two sources at distance $2$, which we denote by $v_4,v_5$. Their edges are all directed towards a cycle of length $4$. The directions of the edges in this cycle must be directed, for otherwise there would be a sink in this cycle which is not possible. We denote this oriented cycle by $v_0 \to v_1 \to v_2 \to v_3 \to v_0$, and we must have
  $$ p_{01} = p_{12} = p_{2,3} = p_{3,0} = 1, $$
  where we used the notation $p_{ij} = p_{v_i,v_j}$, for simplicity. Applying formula \eqref{eq:curvsharptotdeg} to $(x,y)=(v_4,v_0)$ yields
  $$ p_{40} (-2 (p_{01}+p_{03}) + \sum_{i=0}^3 p_{4i}) + p_{41}p_{10} + p_{43}p_{30} = p_{40} (-2 + 1) + p_{43} = 0, $$
  that is $p_{40}=p_{43}$. Similarly, we can show that all transition rates $p_{4i}$ must coincide and, therefore, $p_{40}=p_{41}=p_{42}=p_{43}=1/4$. The same arguments apply to the vertex $v_5$, finishing the proof of the proposition.
\end{proof}

\begin{conj}[Flow limits of the octahedron] The normalized curvature flow on the octahedron for any non-degenerate initial weighting scheme without laziness different from the simple random walk converges always to a limit described in Proposition \ref{prop:octahedron}.
\end{conj}

\section{Asymptotically stable and unstable curvature sharp Markovian weighting schemes}

Recall that every curvature sharp Markovian weighting scheme on a given combinatorial graph is a stationary solution of the normalized curvature flow. It is natural to ask whether such a stationary solution $P^s$ is \emph{asymptotically stable}, that is, whether any closeby Markovian weigthing scheme $P$ converges back to this equilibrium $P^s$ as $t \to \infty$. This can be decided via the linearization of the curvature flow equations around such an equilibrium. In the next subsection, we will describe this linearization in full detail before we consider various examples in the following subsection.

\subsection{Linearization of the curvature flow equations at equilibria}

Let $P^s$ be a curvature sharp Markovian weighting scheme of a finite simple mixed combinatorial graph $G = (V,E)$. We consider Markovian weighting schemes $P$ near $P^s$ with the same laziness, that is $p_{xx} = p_{xx}^s$ for all $x \in V$. 
Let $E^{\rm{dir}} = \{ (x,y) \in V \times V: d(x,y) = 1 \}$. The linearization of $F$ at the equilibrium $P^s$ of the normalized curvature flow is given for each component function $F_{xy}$, $(x,y) \in E^{\rm{dir}}$, by
\begin{multline*}
    DF_{xy}( P^s 
    ) (q_{uv})_{(u,v) \in E^{\rm{dir}}}
    ) = \\
    \left( -4 p_{yx}^s - 2 \sum_{y' \neq y} p_{yy'}^s + \frac{4}{D_x} \sum_{y'} p_{xy'}^s p_{y'x}^s + \frac{1}{D_x}\sum_{y',y''} p_{xy'}^s p_{y'y''}^s - p_{yy}^s \right) q_{xy}\\
    + p_{xy}^s \left( -4 q_{yx} - 2 \sum_{y' \neq y} q_{yy'} + \frac{4}{D_x} \sum_{y'} \left(p_{xy'}^s q_{y'x} + p_{y'x}^s q_{xy'}\right) + \frac{1}{D_x} \sum_{y',y''} \left( p_{xy'}^s q_{y'y''} + p_{y'y''}^s q_{xy'} \right)
    \right) = \\
    \sum_{y' \in S_1(x)} B_{xy}(xy') q_{xy'} + \sum_{y' \in S_1(x)} \sum_{z \in B_1(x)} B_{xy}(y'z) q_{y'z}.
\end{multline*}
Here $y',y''$ are vertices in $S_1(x)$ 
and the potentially non-zero $B$-coefficients are given by
\begin{eqnarray}
B_{xy}(xy) &=& -4 p_{yx}^s - p_{yy}^s - 2 \sum_{y' \neq y} p_{yy'}^s \nonumber \\
&& + \frac{1}{D_x} \left( 4p_{xy}^sp_{yx}^s + 4\sum_{y'} p_{xy'}^s p_{y'x}^s + p_{xy}^s \sum_{y'} p_{yy'}^s + \sum_{y',y''} p_{xy'}^s p_{y'y''}^s \right), \label{eq:Bxy} \\
B_{xy}(xy') &=& \frac{4}{D_x} p_{xy}^s p_{y'x}^s + \frac{1}{D_x} p_{xy}^s \sum_{y''} p_{y'y''}^s + p_{y'y}^s \qquad \qquad\, \text{if $y' \neq y$},  \label{eq:Bxy'} \\
B_{xy}(yx) &=& 4 p_{xy}^s \left( \frac{1}{D_x} p_{xy}^s-1\right), \label{eq:Byx} \\
B_{xy}(y'x) &=& \frac{4}{D_x} p_{xy}^s p_{xy'}^s \qquad \qquad \qquad \qquad \qquad \qquad \qquad \,\,\,\,\,\, \text{if $y' \neq y$}, \label{eq:By'x} \\
B_{xy}(yy') &=& p_{xy}^s \left( \frac{1}{D_x} p_{xy}^s - 2 \right) \qquad \qquad \qquad \qquad \qquad \,\,\,\,\,\,\,\,\,\, \text{if $y' \neq y$}, \label{eq:Byy'} \\
B_{xy}(y'y) &=& p_{xy'}^s\left( \frac{1}{D_x} p_{xy}^s + 1 \right) \qquad \qquad \qquad \qquad \qquad \quad\,\,\, \text{if $y' \neq y$}, \label{eq:By'y} \\
B_{xy}(y'y'') &=& \frac{1}{D_x} p_{xy}^s p_{xy'}^s\qquad \qquad \qquad \qquad \qquad \qquad \qquad \quad \text{if $y' \neq y$ and $y'' \neq y,y'$}. \label{eq:By'y''}
\end{eqnarray}
All other $B$-coefficients are chosen to be zero. Note, however, that the 
transition probabilities $(p_{xy})_{y \in S_1(x)}$ are not independent, and therefore the choice of the $B$-coefficients is not unique, as explained in the following remark.

\begin{rmk} \label{rmk:compl-freedom-asymp}
  Since we have $\sum_{v \in S_1(u)} q_{uv} = 0$ for all $u \in V$, there is a degree of freedom in the choice of the $B$-coefficients. For example, in the case that $G$ is an unmixed complete graph, we can replace $B_{xy}(u,v)$ by $B'_{xy}(u,v) = C_u + B_{xy}(u,v)$ with arbitrary constants $C_u$. This allows us the modify the $B$-coefficients $B_{xy}(yy')$ and $B_{xy}(y'y'')$ in 
  \eqref{eq:Byy'} and \eqref{eq:By'y''} to vanish, and we can use instead
$$ 
  DF_{xy}( P^s 
  ) ( (q_{uv})_{(u,v) \in E^{\rm{dir}}}
  ) = \sum_{y'\neq x} B'_{xy}(xy')q_{xy'} + \sum_{y' \neq x} B'_{xy}(y'x)q_{y'x} + \sum_{y'\neq x,y} B'_{xy}(y'y)q_{y'y}
$$
with
\begin{eqnarray*}
B'_{xy}(xy) &=& -4 p_{yx}^s - p_{yy}^s - 2 \sum_{y' \neq y} p_{yy'}^s + \frac{1}{D_x} \left( 4p_{xy}^sp_{yx}^s + 4\sum_{y'} p_{xy'}^s p_{y'x}^s + p_{xy}^s \sum_{y'} p_{yy'}^s + \sum_{y',y''} p_{xy'}^s p_{y'y''}^s \right), \\
B'_{xy}(xy') &=& \frac{4}{D_x} p_{xy}^s p_{y'x}^s + \frac{1}{D_x} p_{xy}^s \sum_{y''} p_{y'y''}^s + p_{y'y}^s \qquad \qquad\, \text{if $y' \neq x,y$},  \nonumber \\
B'_{xy}(yx) &=& p_{xy}^s \left( \frac{3}{D_x} p_{xy}^s-2\right), \\
B'_{xy}(y'x) &=& \frac{3}{D_x} p_{xy}^s p_{xy'}^s \qquad \qquad \qquad \qquad \qquad \qquad \qquad \,\,\,\,\,\, \text{if $y' \neq x,y$}, \\
B'_{xy}(y'y) &=& p_{xy'}^s \qquad \qquad \qquad \qquad \qquad \qquad \qquad \qquad \qquad \text{if $y' \neq x,y$}. \\
\end{eqnarray*}
Note that in the case of the unmixed complete graph we have $S_1(x) = V \setminus \{x\}$ and, in the formulas for the $B'$-coefficients, $y'y''$ represent all vertices different from $x$, as before. 
\end{rmk}

To end up with a uniquely defined Jacobi matric of $F$, we need to restrict to transition probabilities which are independent. For that we introduce the subset $E^{\rm{ess}} \subset E^{\rm{dir}}$ of ``essential'' transition probabilities by removing, for each $x \in V$ with outgoing directed edges, that is $S_1(x) \neq \emptyset$, one pair $(x,y)$ from $E^{\rm{dir}}$. The cardinality of $E^{\rm{ess}}$ is $M := |E^1| + 2 |E^2| - |V_0|$, where $V_0 \subset V$ is the subset of vertices $x \in V$ for which we have $S_1(x) \neq \emptyset$. Any choice $(p_{uv})_{(u,v) \in E^{\rm{ess}}}$ determines then a weighting scheme $P$ by setting $p_{uv} = D_u - \sum_{v' \in S_1(u) \setminus \{v \}} p_{uv'}$ for the directed edge $(u,v) \not\in E^{\rm{ess}}$. Similarly, any choice $(q_{uv})_{(u,v) \in E^{\rm{ess}}}$
determines also the parameters $q_{uv}$ with $(u,v) \not\in E^{\rm{ess}}$
by setting $q_{uv} = - \sum_{v' \in S_1(u) \setminus \{v \}} q_{uv'}$. Then
$DF((p_{uv}^s)_{(u,v) \in E^{\rm{ess}}})$ is a square matrix of size $M$, and
the weighting scheme $P^s$ corresponding to $(p_{uv}^s)_{(u,v) \in E^{\rm{ess}}}$ is \emph{asymptotically stable} if and only if the real parts of all eigenvalues of this square matrix are negative, and the weighting scheme $P^s$ is \emph{unstable} if and only if at least one of the real parts of these eigenvalues is positive.

Let us reformulate this restriction in terms of matrix multiplications. We start by enumerating the vertices of the graph $G = (V,E)$:
$V = \{ v_0,\dots,v_N \}$. We also introduce the following enumeration on the directed edges in $E^{\rm{dir}}$: Let $1 = j_0$ and
$a_{j_0},\dots,a_{k_0}$ be the edges of the type $(v_0,*)$ (where second vertices are chosen with increasing indices) in $E^{\rm{dir}}$, $j_1 = k_0+1$ and $a_{j_1},\dots,a_{k_1}$ be the edges of the type $(v_1,*)$ in $E^{\rm{dir}}$, and so on. For all vertices $v_l \in V$ with $S_1(v_l) = \emptyset$, we set $j_l = k_{l-1}+1$ and $k_l = k_{l-1}$.
We remove the edges $a_{k_0}, a_{k_1}, \dots, a_{k_N}$ from $E^{\rm{dir}}$ to obtain $E^{\rm{ess}}$. 
For simplicity, we use the notation
$p_j^s$ and $q_j$ for $p_{a_j}^s$ and $q_{a_j}$, and we can write for all $a_j \in E^{\rm{ess}}$,
$$ DF_{a_j}((p_k^s)_{a_k \in E^{\rm{ess}}})((q_k)_{a_k \in E^{\rm{ess}}} ) = \sum_{l=0}^{N}
\sum_{k=j_l}^{k_l} B_{a_j}(a_k) q_k = \sum_{l=0}^{N}
\sum_{k=j_l}^{k_l-1} (B_{a_j}(a_k)-B_{a_j}(a_{k_l})) q_k,
$$
where the last expression involves only parameters $q_k$ corresponding to essential directed edges $a_k \in E^{\rm{ess}}$.
Consequently, the Jacobi matrix $DF((p_k^s)_{a_k \in E^{\rm{ess}}})((q_k)_{a_k \in E^{\rm{ess}}} )$ can be written as
\begin{equation} 
DF((p_k^s)_{a_k \in E^{\rm{ess}}}) = P_1 B P_2 
\end{equation}
where $B$ is the square matrix of size $k_N$ with $B_{jk} = B_{a_j}(a_k)$ for $a_j,a_k \in E^{\rm{dir}}$, $P_1$ is obtained from the identity matrix $I_{k_N}$ by removing the rows $k_0, k_1,\dots, k_N$, and $P_2 = P_1^\top - P_3$ with $P_3$ a $k_N \times M$ matrix whose first $k_0-1$ columns are all the standard basis vector $e_{k_0}$, the next $k_1-j_1$ columns are all the standard basis vector $e_{k_1}$, and so on.

\subsection{Examples of asymptotically stable and unstable equilibria}

In this subsection we investigate curvature sharp weighting schemes of various examples of unmixed combinatorial graphs.

\begin{ex}[Curvature sharp weighting schemes on a cycle] Let $C_N = (V,E)$ be a cycle of length $N \ge 4$, that is $V= \{v_0,v_1,\dots,v_{N-1}\}$ and $v_i \sim v_{i+1}$ with indices $i$ modulo $N$. For simplicity, we refer to vertex $v_i$ henceforth as $i$. We assume $P^s$ to be a non-lazy curvature sharp weighting scheme on $C_N$, and we remove the directed edges $(i,i+1) \in E^{\rm{dir}}$ to obtain $E^{\rm{ess}}$, and the only coefficients $B_{a_j}(a_k)$ with $a_j \in E^{\rm{ess}}$ and $a_k \in E^{\rm{dir}}$, which may be potentially non-zero are (see \eqref{eq:Bxy}, \eqref{eq:Bxy'}, \eqref{eq:Byx} and \eqref{eq:By'x}),
\begin{eqnarray*}
B_{i,i-1}(i,i-1) &=& -4p_{i,i-1}^s + 8p_{i,i-1}^s p_{i-1,i}^s + 4p_{i,i+1}^s p_{i+1,i}^s, \\
B_{i,i-1}(i,i+1) &=& 4 p_{i,i-1}^s p_{i+1,i}^s, \\
B_{i,i-1}(i-1,i) &=& 4(p_{i,i-1}^s)^2-4p_{i,i-1}^s, \\
B_{i,i-1}(i+1,i) &=& 4p_{i,i-1}^s p_{i,i+1}^s.
\end{eqnarray*}
This implies
\begin{multline*}
DF_{i,i-1}((p_k^s)_{a_k \in E^{\rm{ess}}})((q_k)_{a_k \in E^{\rm{ess}}}) \\
= (B_{i,i-1}(i,i-1)-B_{i,i-1}(i,i+1)) q_{i,i-1} + B_{i,i-1}(i+1,i) q_{i+1,i} - B_{i,i-1}(i-1,i) q_{i-1,i-2} \\
= 4( (p_{i,i-1}^s)^2-p_{i,i-1}^s ) q_{i-1,i-2} + 4(-p_{i,i-1}^s + 2 p_{i,i-1}^s p_{i-1,i}^s + p_{i,i+1}^s p_{i+1,i}^s - p_{i,i-1}^s p_{i+1,i}^s) q_{i,i-1}
+ 4 p_{i,i-1}^s p_{i,i+1}^s q_{i+1,i}.
\end{multline*}
In the case of the simple random walk $p_{i,i_1}^s=p_{i,i+1}^s = 1/2$, $i \in \{0,\dots,N-1\}$, this simplifies to
$$ DF_{i,i-1}((p_k^s)_{a_k \in E^{\rm{ess}}})((q_k)_{a_k \in E^{\rm{ess}}}) = q_{i-1,i-2} + q_{i+1,i},$$
and the corresponding matrix $DF((p_k^s)_{a_k \in E^{\rm{ess}}})$ coincides with the adjacency matrix of $C_N$ whose largest eigenvalue is $2$. Therefore the simple random walk on $C_N$, $N \ge 4$, is an unstable equilibrium. This is in contrast to the simple random walk on $C_3 = K_3$, which is asymptotically stable, as we will see in Example \ref{ex:srwcompletestable}. 

In the case of the totally degenerate clockwise weighting scheme $p_{i,i-1}=1$ and $p_{i,i+1}=0$ (see right hand side of Figure \ref{fig:path-cycle-graph}), we have
$$ DF_{i,i-1}((p_k^s)_{a_k \in E^{\rm{ess}}})((q_k)_{a_k \in E^{\rm{ess}}}) = - 4 q_{i,i-1}, $$
that is, $DF((p_k^s)_{a_k \in E^{\rm{ess}}}) = - 4 {\rm{Id}}_N$, and this curvature sharp Markovian weigthing scheme is asymptotically stable. This agrees with the fact that many initial weighting schemes end up in this limit under the curvature flow. The same holds true for the corresponding totally degenerate anti-clockwise weighting scheme with $p_{i,i+1}=1$ and $p_{i,i-1}=0$.
\end{ex}

\begin{ex}[Flow limits of the octahedron]
We know from Example \ref{ex:octahedron} that stationary solutions of the normalized curvature flow on the octahedron are the simple random walk without laziness as well as the totally degenerate weighting scheme given as matrix $P^s$ in the following code (see also the right hand side of Figure \ref{fig:octahedron}). The linearization $DF(P^s)$ is analyzed in line 10 of the program, and the function \blue{\texttt{equilibrium_type}} with the parameters chosen in lines 6 and 7 returns one of the values $-1,0,1$ (corresponding to ``asymtotically stable'', ''undecided'', ``unstable'', respectively), following by a list of its eigenvalues. That is, after execution of line 10, {\tt{result[0]}} is one of the values $-1,0,1$ and {\tt{result[1]}} is a list of the 18 eigenvalues $\lambda_j$ of $DF(P^s)$.

\begin{lstlisting}[language=Python]{}
A = [[0,1,0,1,1,1],[1,0,1,0,1,1],[0,1,0,1,1,1],
     [1,0,1,0,1,1],[1,1,1,1,0,0],[1,1,1,1,0,0]]
# Ps = srw(A)
Ps = [[0,1,0,0,0,0],[0,0,1,0,0,0],[0,0,0,1,0,0],
      [1,0,0,0,0,0],[1/4,1/4,1/4,1/4,0,0],[1/4,1/4,1/4,1/4,0,0]]
eigenvalues = True
jacobi_matrix = False
threshold = 0.001
norm_tolerance = 0.001
result = equilibrium_type(A,Ps,eigenvalues,jacobi_matrix,norm_tolerance,threshold)
print("Flow dynamics eigenvalues:")
for j in range(18):
    print(np.around(result[1][j],3))
\end{lstlisting}


The program provides us with the following list of complex eigenvalues:

\medskip

\begin{center}
\begin{tabular}{lrrrrrrrr}
$\lambda_j$ & $-1$ & $-1+i$ & $-1-i$
& $-2$ & $-2+i$ & $-2-i$ & $-3$ & $-4$ \\ \hline
multiplicity &
$2$ & $2$ & $2$ & $2$ & $2$ & $2$ & $2$ & $4$ \end{tabular}
\end{center}

\medskip

This shows that the totally degenerate curvature sharp weighting scheme $P^s$ on the octahedron is asymptotically stable. This is expected since the initial weighting scheme $P_0$ in \eqref{eq:P0-octahedron} of Example \ref{ex:octahedron} converges to this limit under the normalized numerical curvature flow. 

Running the same code for the simple random walk without laziness instead (by uncommenting line 3 and commenting out lines 4 and 5 in the above code) shows that this second curvature sharp weighting scheme is unstable. The eigenvalues in this case are all real valued, one of them $0.5$, and given as follows:

\medskip

\begin{center}
\begin{tabular}{lrrrrr}
$\lambda_j$ & $0.5$ & $0$ & $-0.75$
& $-1$ & $-1.5$ \\ \hline
multiplicity &
$3$ & $3$ & $2$ & $6$ & $4$
\end{tabular}
\end{center}
\end{ex}

\begin{ex}[Simple random walk on a complete graph] \label{ex:srwcompletestable} Let $K_{n+1}=(V,E)$ be the complete unmixed graph with $n+1\ge 3$ vertices. 
Instead of the $B$-coefficients, we make use of the $B'$-coefficients introduced in Remark \ref{rmk:compl-freedom-asymp}. The simple random walk $p_{xy} = \frac{1}{n}$ for $x \neq y$ is a curvature sharp Markovian weighting scheme, and the non-zero $B'$-coefficients are then given by
$$ B'_{xy}(xy) = \frac{(n+1)(3-n)}{n^2}, \, B'_{xy}(xy') = \frac{2n+3}{n^2}, \,
B'_{xy}(yx) = \frac{3-2n}{n^2}, \, B'_{xy}(y'x) = \frac{3}{n^2}, \, B'_{xy}(y'y) = \frac{1}{n}, $$
where $y' \in V$ is an arbitrary vertex different from $x,y$. Let $\{ v_0, v_1, \dots, v_n \}$ be the vertex set of $K_{n+1}$ and, as in the previous example, we refer to vertex $v_i$ as $i$, for simplicity.

Let us now consider the case $n=2$. The process of removing edges from $E^{\rm{dir}}$ described earlier leads to the following remaining
edges in $E^{\rm{ess}}$:
$$ a_1 = (0,1), \quad a_3 = (1,0), \quad a_5 = (2,0). $$
Choosing the simple random walk $p_{01}^s = p_{10}^s = p_{20}^s = 1/2$, we have
\begin{multline}
DF(p_{01}^s,p_{10}^s,p_{20}^s) = \begin{pmatrix} e_1^\top \\ e_3^\top \\ e_5^\top  \end{pmatrix} B' \begin{pmatrix} e_1-e_2 & e_3-e_4 & e_5 - e_6 \end{pmatrix} \\
= \begin{pmatrix}  B_{01}'(01) - B'_{01}(02) & B_{01}'(10) - B_{01}'(12) & B_{01}'(20) - B_{01}'(21) \\
B_{10}'(01) - B_{10}'(02) & B_{10}'(10) - B_{10}'(12) & B_{10}'(20) - B_{10}'(21) \\
B_{20}'(01) - B_{20}'(02) & B_{20}'(10) - B_{20}'(12) & B_{20}'(20) - B_{20}'(21) \end{pmatrix} = \begin{pmatrix} -1 & \frac{1}{4} & \frac{1}{4} \\ \frac{1}{4} & -1 & \frac{1}{4} \\
\frac{1}{4} & \frac{1}{4} & -1 \end{pmatrix},
\end{multline}

which is a negative definite matrix with eigenvalues $-\frac{1}{2}$, $-\frac{5}{4}$, $-\frac{5}{4}$, showing that the simple random walk on $K_3$ is an asymptotically stable 
equilibrium. This result can be numerically verified via the following code:

\begin{lstlisting}[language=Python]{}
n = 2
A = complete(n+1)
P = srw(A)
result = equilibrium_type(A, P, True, True)
print()
print("Eigenvalues:")
print(np.around(result[1], 3))
print()
print("Linearised flow matrix at equilibrium:")
print(np.around(result[2], 3))
\end{lstlisting}

The program returns the eigenvalues of the linearized flow matrix of $K_3$ at the simple random walk. Note however that the computed matrix is based here on the $B$-coefficients instead of the $B'$-coefficients, so this matrix is slightly different from the one given above, while the eigenvalues are the same.

In the case $n=3$, 
we have $p_{01}^s=p_{02}^s=p_{10}^s=p_{12}^s=p_{20}^s=p_{21}^s=p_{30}^s=p_{31}^s=1/3$ and the returned matrix by the program
(after changing {\tt{n=2}} into {\tt{n=3}} in line 1 of the code) is as follows:

$$
     DF(p_{01}^s,p_{02}^s,p_{10}^s,p_{12}^s,p_{20}^s,p_{21}^s,p_{30}^s,p_{31}^s) = \begin{pmatrix} 
     -1 & 0 & -1/3 & 0 & 1/3 & 1/3 & 1/3 & 1/3 \\ 
     0 & -1 & 1/3 & 1/3 & -1/3 & 0 & 0 & -1/3 \\
     -1/3 & 0 & -1 & 0 & 1/3 & 1/3 & 1/3 & 1/3 \\
     1/3 & 1/3 & 0 & -1 & 0 & -1/3 & -1/3 & 0 \\
     0 & -1/3 & 1/3 & 1/3 & -1 & 0 & 0 & -1/3 \\
     1/3 & 1/3 & 0 & -1/3 & 0 & -1 & -1/3 & 0 \\
     1/3 & 1/3 & 0 & -1/3 & 0 & -1/3 & -1 & 0 \\
     0 & -1/3 & 1/3 & 1/3 & -1/3 & 0 & 0 & -1
     \end{pmatrix}.
$$ 
Its eigenvalues are $-\frac{2}{3}$ with multiplicity $6$ and $-2$ with multiplicity $2$. This shows that the simple random walk on $K_4$ is again an asymptotically stable equilibrium.
\end{ex}

Since numerical experiments show that all non-degenerate Markovian weighting schemes without laziness on $K_{n+1}$ converge to the simple random walk, we expect that the simple random walk on $K_{n+1}$ is an asymptotically stable equilibrium for all $n \ge 2$ (see also our Conjecture \ref{conj:complgraph}, which an is even stronger statement). To this end, our code provides the following numerical results: The eigenvalues of $DF(P^s)$ for the simple random walk without laziness on the complete graph $K_{n+1}$ are real valued and given by
\begin{itemize}
    \item $-\frac{7}{16}$ with multiplicity $4$, $-\frac{3}{4}$ with multiplicity $6$, and $-\frac{7}{4}$ with multiplicity $5$ for $n=4$,
    \item $-\frac{8}{25}$ with multiplicity $5$, $-\frac{4}{5}$ with multiplicity $10$, and
    $-\frac{8}{5}$ with multiplicity $9$ for $n=5$,
    \item $-\frac{1}{4}$ with multiplicity $6$, $-\frac{5}{6}$ with multiplicity $15$, and $-\frac{3}{2}$ with multiplicity $14$ for $n=6$, and
    \item $-\frac{10}{49}$ with multiplicity $7$, $- \frac{6}{7}$ with multiplicity $21$, and $-\frac{10}{7}$ with multiplicity 20 for $n=7$.
\end{itemize}
These results give rise to the following conjecture:

\begin{conj}
 The non-lazy simple random walk $P^s$ on the unmixed complete graph $K_{n+1}$, $n \ge 2$, is asymptotically stable and the eigenvalues of $DF(P^s)$ are given by
 \begin{eqnarray*}
 - \frac{n-1}{n} && \qquad \text{with multiplicity ${n \choose 2}$,} \\
 - \frac{n+3}{n^2} && \qquad \text{with multiplicity $n$,} \\
 - \frac{n+3}{n} && \qquad \text{with multiplicity ${n \choose 2}-1$.}
 \end{eqnarray*}
 As $n \to \infty$, we have $- \frac{n-1}{n} \to -1$, $- \frac{n+3}{n^2} \approx -\frac{1}{n} \to 0$ and $- \frac{n+3}{n} \to -1$.
\end{conj}

\begin{ex}[Simple random walks on hypercubes]
We know from \cite[Corollary 1.14]{CKLMPS-22} that the simple random walk without laziness on a hypercube $Q^d = (K_2)^d$ is curvature sharp, and it is the only non-degenerate curvature sharp weigthing schemes without laziness if and only if $d$ is odd. For $d=2$, a non-degenerate curvature sharp weigthing scheme on $Q^2$ different from the simple random walk is given in line 4 of the following code:
\begin{lstlisting}[language=Python]{}
A = hypercube(2)
p = rand()
q = 1-p
P = [[0,p,q,0],[p,0,0,q],[p,0,0,q],[0,p,q,0]]
result = equilibrium_type(A,P,True)
print("flow dynamics eigenvalues:", result[1])
\end{lstlisting}
All eigenvalues $\lambda_j$ of the corresponding linearized flow matrix provided by this code via line 6 seem to be real. Two of them are numerically zero and the other two are given by $\pm \lambda$ with a non-zero real $\lambda$. So this equilibrium is unstable.

Our experiments for arbitrary weighting schemes on $Q^d$, $d \ge 2$, close to the simple random walk show that the numerical flow converges usually to a degenerate limit. This agrees with our observation that the linearized flow matrix at the simple random walk seems always to have real eigenvalues with some of them being positive. These eigenvalues can be obtained via the following code:
\begin{lstlisting}[language=Python]{}
d = 3
A = hypercube(d)
P = srw(A)
result = equilibrium_type(A,P,True)
print("Flow dynamics eigenvalues at the simple random walk:")
for j in range((d-1)*(2**d)):
    print(np.around(result[1][j],4))
\end{lstlisting} 
\end{ex}

Let us finish this section with a conjecture which we verified numerically for $d=2,3,\dots,9$:

\begin{conj}
  The non-lazy simple random walk $P^s$ on the hypercube $Q^d$, $d \ge 2$, is unstable and the eigenvalues of the corresponding linearized curvature flow matrix are all real with the largest eigenvalue $\lambda_{\max}$ equals
  $$ \lambda_{\max} = \frac{4}{d}. $$
\end{conj}

\section{Implementation of the curvature flow and useful tools}

This section provides a short description of a program, written in Python, which allows users to carry out their own curvature flow experiments. This program is provided as an ancillary file and can be used in combination with the ``Graph Curvature Calculator'', 
which can be freely accessed at
\begin{center}
\blue{\url{http://www.mas.ncl.ac.uk/graph-curvature}} 
\end{center}
The Graph Curvature Calculator is a powerful but easy to use interactive Web tool to draw graphs and to compute various types of curvatures like Bakry-\'Emery curvature on its vertices or Ollivier Ricci curvature on its edges (for details see \cite{CKLLS-19}). This Web tool allows users to obtain the adjacency matrix of the graph under consideration which can then be used as input for the curvature flow code. 

\smallskip

The curvature flow code provides functions and routines which can be divided into six categories: 
\begin{enumerate}
    \item Functions related to graphs and their adjacency matrices
    \item Functions related to weighting schemes
    \item Functions testing combinatorial and weighted graphs
    \item Curvature flow computation routines
    \item Curvature computation routines
    \item Display routines
\end{enumerate}

Each category is discussed in one of the following subsections.

\subsection{Functions related to graphs and their adjacency matrices}

Recall that the topology of a given Markov chain $(G,P)$ is contained in the combinatorial graph $G=(V,E)$. This information is provided by the adjacency matrix $A = A_G$ of the graph. Firstly, we go through some useful functions performing operations with these adjacency matrices. Every graph generated by one of the functions below is returned as such an adjacency matrix, in particular as a NumPy array.  The names of these functions are as follows:

\begin{itemize}
    \item \blue{\texttt{rand\_adj\_mat(n, p, connected=False)},}
    \item \blue{\texttt{complete(n)},}
    \item \blue{\texttt{path(n)},}
    \item \blue{\texttt{cycle(n)},}
    \item \blue{\texttt{wedge\_sum(A, B, i, j)},}
    \item \blue{\texttt{bridge\_at(A, B, i, j)},}
    \item \blue{\texttt{hypercube(n)},}  
    \item \blue{\texttt{cart\_prod(A, B)},} 
    \item \blue{\texttt{onespheres(A)}.}
\end{itemize}

\blue{\texttt{rand\_adjmat(n, p)}} returns a random adjacency matrix with \texttt{n} vertices, where \texttt{p} is the probability that an edge exists between any two vertices. Therefore higher values of \texttt{p} usually lead to better connected graphs.
Choosing \texttt{connected=True} instead guarantees that the returned adjacency matrix provides a connected combinatorial graph.

The functions \blue{\texttt{complete(n)}}, \blue{\texttt{path(n)}} and \blue{\texttt{cycle(n)}} return the adjacency matrix of a complete graph with \texttt{n} vertices, a path of length $n$ and a cycle of length $n$, respectively.

\blue{\texttt{wedge\_sum(A, B, i, j)}} produces an adjacency matrix of a graph that is the wedge sum of two pointed graphs represented by the adajcency matrices \texttt{A} and \texttt{B} and the \texttt{i}th vertex of \texttt{A} and the \texttt{j}th vertex of \texttt{B}, that is, the new graph is obtained from the disjoint union of these two graphs by only identifying these two vertices and keeping all other edges and vertices disjoint.

There is also \blue{\texttt{bridge\_at(A, B, i, j)}}, which returns an adjacency matrix of the graph formed by connecting the graphs represented by \texttt{A} and \texttt{B} by a single edge at the \texttt{i}th vertex of \texttt{A} and the \texttt{j}th vertex of \texttt{B}.

The function \blue{\texttt{hypercube(n)}} returns the adjacency matrix of the \texttt{n}-dimensional hypercube.

The Cartesian product of two graphs represented by adjacency matrices \texttt{A} and \texttt{B} is returned by the function \blue{\texttt{cart\_prod(A, B)}}.
(The $n$-dimensional hypercube cand also be alternatively generated as the Cartesian product of $n$ complete graphs $K_2$.)

Finally, for a graph $G = (V,E)$ given by the adjacency matrix \texttt{A} and $V = \{0,1,\dots,n-1\}$, the function \blue{\texttt{onespheres(A)}} returns a list of lists whose $i$-th entry is a list of all neighbours of vertex $i \in V$, and whose $n$-th entry is a list of the combinatorial degrees of the vertices in $V$. This function is mainly used in execution of the other functions.

\subsection{Functions related to weighting schemes}

The weighting scheme of a Markov chain $(G,P)$ is provided via the weighted matrix $P = P_G$. The functions in this category are the following:

\goodbreak

\begin{itemize}
    \item \blue{\texttt{randomizer(A, threshold=0.001, laziness=False)},}
    \item \blue{\texttt{srw(A, laziness=False)},}
    \item \blue{\texttt{cart\_prod\_prob(P, Q, p, q)}.}
\end{itemize}

\blue{\texttt{randomizer(A)}} and \blue{\texttt{srw(A)}} are two useful functions that can be used to give weighted matrices from an adjacency matrix.

The function \blue{\texttt{randomizer(A)}} returns a weighting scheme for the graph $G$ with random, numerically non-degenerate transition rates, that is, no transition rate is chosen to be below the parameter 
\texttt{threshold}. Usually, the returned weighting schemes are without laziness, but choosing \texttt{laziness=True} returns weighting schemes with all vertices having laziness $\ge$ \texttt{threshold}.

\blue{\texttt{srw(A)}} returns the weighting scheme corresponding to the non-lazy simple random walk on the graph $G = (V,E)$ represented by the adjacency matrix \texttt{A}. Choosing \texttt{laziness=True} the transition rates of a vertex $v \in V$ with degree $n$ to its neighbours are chosen to be $\frac{1}{n+1}$ and its laziness is also chosen to be $\frac{1}{n+1}$. 

The function \blue{\texttt{cart\_prod\_prob(P, Q, p ,q)}} is a ``weighted'' analogue of the function \blue{\texttt{cart\_prod}} from the previous subsection for the Cartesian product of two weighting schemes \texttt{P}, \texttt{Q} with weights \texttt{p}, \texttt{q}, with $\texttt{p}+\texttt{q} = 1$. If \texttt{P} and \texttt{Q} are of size $n$ and $m$, respectively, this function returns the matrix $\texttt{p} \texttt{P} \otimes I_n + \texttt{q} I_m \otimes \texttt{Q}$ of size $nm$, where $I_n$ is the identity matrix of size $n$ and $A \otimes B$ is the Kronecker product of $A$ and $B$.    

\subsection{Functions testing combinational and weighted graphs}

For combinatorial graphs given by their adjacency matrices \texttt{A} and weighted graphs given additionally by their weighting scheme \texttt{P}, we have the following test functions:

\begin{itemize}
    \item \blue{\texttt{is\_connected(A)},}
    \item \blue{\texttt{is\_weakly\_connected(A, threshold=0.001)},}
    \item \blue{\texttt{is\_totally\_degenerate(A, P, threshold=0.001)},}
    \item \blue{\texttt{is\_markovian(P, norm\_tolerance=0.001)},}
    \item \blue{\texttt{is\_curvature\_sharp(A, P, norm\_tolerance=0.001, threshold=0.001)},}
    \item \blue{\texttt{equilibrium\_type(A, P, eigenvalues=False, jacobi\_matrix=False,\\ norm\_tolerance=0.001, threshold=0.001)}.}
\end{itemize}

The function \blue{\texttt{is\_connected(A)}} returns \texttt{True} if and only if the adjacency matrix \texttt{A} represents a connected graph. 

The function \blue{\texttt{is\_weakly\_connected(P)}} is a ``weighted'' analogue, which returns \texttt{True} if and only if the weighted matrix \texttt{P} represents a weakly connected graph. It does this by forming an adjacency matrix with a \texttt{1} in the $(\texttt{i}, \texttt{j})$th entry if and only if \texttt{P[i, j] > threshold} or \texttt{P[j, i] > threshold} and then testing for connectedness. 

Recall that a weighted graph $(G,P)$ is called \emph{numerically totally degenerate} if there are no two-sided edges with numerical non-zero transition rates in both directions and no one-sided edges with numerical non-zero transition rate, where we consider a transition rate $p_{xy}$ as numerically non-zero if and only if $p_{xy} \ge$ \texttt{threshold}. The function \blue{\texttt{is\_totally\_degenerate(A,P)}} tests this property. 

The function \blue{\texttt{is\_markovian(P)}} tests whether the entries of each of the columns of $P$ add up numerically to $1$ up to an error $\le \texttt{norm\_tolerance}$. 

Numerical curvature sharpness (up to an error $\le \texttt{threshold}$) of Markovian weighted graphs given by \texttt{(A,P)} is tested by \blue{\texttt{is\_curvature\_sharp(A,P)}}. If \texttt{(A,P)} fails to be Markovian (with respect to \texttt{norm\_tolerance}), this function returns \texttt{NONE} and gives notice to the user by an error message. 

For dynamical investigations of curvature flow equilibria, we have the function \blue{\texttt{equilibrium\_type(A,P)}} which always returns a list of length three. The function checks first whether \texttt{(A,P)} satisfies the Markovian property and is numerically curvature sharp. If this is not the case, it returns a list of three \texttt{NONE} values. Otherwise, the function investigates the real parts of the eigenvalues $\lambda_j$ of the linearized curvature flow matrix at the equilibrium \texttt{P}. The first entry of the return list is $-1, 0$ or $1$ depending on the maximum $\max_j {\rm{Re}}(\lambda_j)$.
If this maximum is $\ge \texttt{threshold}$, the return value is $1$ (for ``unstable'') and if this maximum is $\le - \texttt{threshold}$, the return value is $-1$ (for ``asymptotically stable''). Otherwise, the dynamical nature of the equilibrium cannot be numerically decided and the function returns the value $0$. The following two entries of the return list are usually \texttt{NONE} unless the user made the choices \texttt{eigenvalues=True} or \texttt{jacobi\_matrix=True}. In the first case, the second entry of the return list is a list of all eigenvalues of the linearized curvature flow matrix, and in the second case the third entry of the return is is the linearized curvature flow matrix itself. 

\subsection{Curvature flow computation routines}

At the heart of the program are the curvature flow routines solving the initial value ordinary differential equations \eqref{eq:laziness} and \eqref{eq:flowdiffeq}. The relevant routines are the following:

\begin{itemize}
    \item \blue{\texttt{curv\_flow(A, P, t\_max, dt=0.3, C=zeroes)},}
    \item \blue{\texttt{norm\_curv\_flow(A, P, t\_max, dt=0.3, stoch\_corr=True, norm\_tolerance=0.001)},}
    \item \blue{\texttt{norm\_curv\_flow\_lim(A, P, dt=0.3, stoch\_corr=True, norm\_tolerance=0.001,\\ lim\_tolerance=0.001, t\_lim=10000)}.}
\end{itemize}

The initial Markov chain $(G,P_0)$ with $G=(V,E)$ is entered by the adjacency matrix \texttt{A} describing the topology of the graph $G$ and the weighting scheme \texttt{P} containing the initial probability transitions $p_{xy}(0)$. 

The first routine \blue{\texttt{curv_flow(A,P)}} computes the non-normalized numerical curvature flow with coefficients $C_x(t)=0$ in \eqref{eq:flowdiffeq}. If users decide to choose other coefficient functions, they need to modify the input parameter \texttt{C=zeroes}. Note that \blue{\texttt{zeroes(A,P)}} is a function returning simply a list of zeroes of length $|V|$. 
Users can investigate modifications of the curvature flow by choosing their own coefficient functions with input values \texttt{(A,P)} and returning a list of length $|V|$ of real values. Using the discretization parameter \texttt{dt=0.3} for the discrete time steps starting at \texttt{t=0}, the curvature flow routine creates a list \blue{\texttt{P\_list}} of weighting schemes (represented by NumPy arrays of size $|V| \times |V|$) at each time increment using the Runge-Kutta algorithms RK4. There are two internal subroutines involved which we would like to mention briefly. \blue{\texttt{Pvecs\_to\_P}} translates a weighting scheme given by a list of lists (where each inner list contains the transition rates of the corresponding vertex) into the corresponding NumPy array. \blue{\texttt{Pvecs_prime}} computes, for a given weighting scheme \texttt{P}, the right hand side of the ordinary differential equations describing the curvature flow. The representation of this right hand side is again a list of lists, as described before. The computation stops just before the discrete time steps exceed the limit time \texttt{t\_max} and the routine returns \texttt{P\_list}. Note that in this general setting, transition rates can assume arbitrary values and even negative ones or diverge to infinity in finite time which may lead to system error messages. Users need to be aware of this possibility.

The normalized curvature flow, using the coefficients $C_x(t) = K_{P(t),\infty}^{d(x,\cdot)}(x)$ (see \eqref{eq:CxMark}), is numerically computed by the routine \blue{\texttt{norm\_curv\_flow(A,P)}}. This special flow is the main focus of this paper and could also be mimicked by choosing \texttt{C=K\_inf} in \blue{\texttt{curv\_flow}}. The function \blue{\texttt{K\_inf(A,P)}} returns a list of upper curvature bounds for all vertices of the Markovian weighted graph represented by \texttt{(A,P)} (see \cite[formula (66)]{CKLMPS-22} for an explicit expression of this function in terms of transition rates). During the numerical computations of subsequent time steps, the Markovian property of the corresponding weighting schemes may be slightly violated. If this violation exceeds the threshold \texttt{norm\_tolerance}, the routine prepares for a potential correction according to the Boolean variable \texttt{stoch_corr}. If \texttt{stoch\_corr=False}, the program stops with a message to the user as discussed in Example \ref{ex:random-graph}. Otherwise the program carries out the following automatic Markovian renormalization of the currently considered weighting scheme: while the diagonal entries (the laziness values) are unchanged, the off-diagonal entries of every row are rescaled by the same factor to guarantee that the resulting matrix becomes stochastic again. As before, this routine returns a list \texttt{P\_list} of consecutive weighting schemes up to the time limit
\texttt{t\_max}.

While the user needs to specify the time limit \texttt{t\_max} in the above two routines, the third routine
\blue{\texttt{norm\_curv\_flow\_lim(A,P)}} continues computing the normalized numerical curvature flow until a numerical flow limit is reached. This limit is determined by the parameter \texttt{lim\_tolerance}. The details for this numerical limit are explained in the introductory part of Section \ref{sec:curv-flow-ex}. 
This routine returns a list of length two: the limiting weighting scheme as a NumPy array followed by the numerical convergence time. Since it may happen that a normalized numerical flow does not converge at all (even though we are not aware of any such example), the parameter \texttt{t\_lim} provides an upper time limit beyond which the routine will not continue. The parameters \texttt{stoch\_corr} and \texttt{norm\_tolerance} play the same role as in the routine \blue{\texttt{norm\_curv\_flow}}.

\subsection{Curvature computation routines}

The functions in this section calculate Bakry-\'Emery curvatures and curvature upper bounds of graphs with given weighting schemes at all vertices.

\begin{itemize}
    \item \blue{\texttt{curvatures(A, P, N=inf, onesps=[], q=None)},}
    \item \blue{\texttt{calc\_curvatures(A, P\_list, N=inf, k=1)},}
    \item \blue{\texttt{K\_{inf}(A, P)},}
    \item \blue{\texttt{calc\_curv\_upper\_bound(A, P\_list, N=inf, k=1)}.}
\end{itemize}

The routine \blue{\texttt{curvatures(A,P)}} computes, for a weighted graph $(G,P)$ with $G=(V,E)$ and represented by \texttt{(A,P)}, the curvatures of all vertices for dimension \texttt{N} $=\infty$ and returns them as a list of length $|V|$. If users are interested in curvatures for other dimensions, they need to change the parameter \texttt{N=inf}. There are two other inputs which can speed up the curvature calculations: if the number \texttt{q} of vertices in $V$ is given, it can be specified to avoid its repeated recalculation, for example during a curvature flow process. Similarly, if \blue{\texttt{onespheres(A)}} has already be calculated earlier, this information can be communicated to the routine via the input variable \texttt{onesps}.

After a curvature flow computation with corresponding list \texttt{P\_list} of consecutive weighting schemes, the routine \blue{\texttt{calc\_curvatures(A,P\_list)}} computes the corresponding evolution of vertex curvatures by calling \blue{\texttt{curvatures(A,P\_list[j])}} and returns it as a list of lists. Here the $j$-th inner list contains the curvature evolution of the $j$-th vertex of the graph $G=(V,E)$ represented by \texttt{A}. The dimension parameter \texttt{N=inf} plays the same role as before. \blue{\texttt{calc\_curvatures}} computes curvatures only of each \texttt{k}-th weighting scheme provided by \texttt{P\_list}. Where appropriate, this can help to reduce computation time.
 
The routine \blue{\texttt{K\_inf(A,P)}} was already discussed in the previous subsection and provides a list of upper curvature bounds for all vertices of the weighted graph $(G,P)$ represented by \texttt{(A,P)}. 

\blue{\texttt{calc\_curv\_upper\_bound}} is completely analogous to \blue{\texttt{calc\_curvatures}},  but it calls \blue{\texttt{K\_inf}} instead of \blue{\texttt{curvatures}}.
 
\subsection{Display routines}

The main display routines for users are the evolution of curvatures at various vertices during the curvature flow, the evolution of transition rates of edges emanating from vertices and the display of individual weighted graphs with vertices arranged in a circle.  The relevant routines are the following:

\begin{itemize}
    \item \blue{\texttt{display\_curvatures(curv, dt=0.3, is\_Markovian=True, N=inf, k=1,\\ curv\_bound=[], vertex\_list=[])},}
    \item \blue{\texttt{display\_trans\_rates(A, P\_list, dt=0.3, vertex\_list=[])},}
    \item \blue{\texttt{display\_weighted\_graph(A, P, title=None, threshold=10**(-3), \\display\_options=[10, True, 2, []], laziness=False)}.} 
\end{itemize}

Given the evolution of vertex curvatures during a curvature flow process via a list of lists, where the $j$-th inner list is the curvature evolution of the $j$-th vertex, \blue{\texttt{display_curvatures(curv)}} displays the curvature evolution for each consecutive vertex separately, as illustated for example, in Figure \ref{fig:random-graph-curvatures}. If this information should be only given for specific vertices, this can be specified by the input parameter \texttt{vertex\_list}. The time step \texttt{dt=0.3} and the value of \texttt{k} together determine the labelling of the horizontal time axis. For the role of \texttt{k} we refer readers to our explanation about the routine \blue{\texttt{calc_curvatures}}. Upper curvature bounds can be inserted into the displays by the input parameter \texttt{curv\_bound}, which needs to be given in the same format as the vertex curvatures. Constant lower and upper curvature bounds $-1$ and $2$ are plotted alongside if the Boolean \texttt{is_Markovian} is chosen to be \texttt{True} and if the dimension parameter \texttt{N} is $\ge 2$. These bounds appear, for example, in the illustrations given in Figure \ref{fig:random-graph-curvatures}.

Given the evolution of weighting schemes during a curvature flow process on a graph with adjacency matrix \texttt{A} by a list \texttt{P\_list} of NumPy arrays, \blue{\texttt{display\_trans\_rates(A,P\_list)}} displayes the evolution of transition rates of emanating edges for each consecutive vertex separately, as illustrated for example, in Figure \ref{fig:random-graph-trans-rates}. The input parameters \texttt{dt} and \texttt{vertex\_list} play the same role as in the previous routine.

Finally, there is the routine  \blue{\texttt{display\_graph(A,P)}} with input parameters \texttt{A} and \texttt{P} representing a weighted graph $(G,P)$. This routine produces a MatPlotLib plot of this weighted graph, with the vertices arranged counter-clockwise in a circle. The plot uses the following convention to illustrate different types of edges:
\begin{itemize}
    \item \green{green, solid lines} represent numerically non-degenerate edges, that is $\{x, y\} \in E$ with both $p_{xy}, p_{yx} \ge$ \texttt{threshold},
    \item \red{red, dashed lines} with an arrow represent numerically degenerate edges, that is $\{x,y\} \in E$ with exactly one of $p_{xy}$ and $p_{yx}$ strictly less than \texttt{threshold}.
    \item black, dotted lines represent edges with numerically vanishing transition rates in both directions, that is $\{x,y\} \in E$ with both $p_{xy}, p_{yx}$ strictly less than \texttt{threshold}.
\end{itemize}

Users can add a title into the display by specifying the input parameter \texttt{title}. For Markovian weighted graphs with non-vanishing laziness, the option \texttt{laziness=True} labels each vertex with its corresponding laziness.  

It remains to discuss the input parameter \texttt{display\_options} of the \blue{\texttt{display\_graph}} routine. This parameter is a list of four entries. The first entry determines the size of the plot. Usually, the transition rates are printed above the edges, but if the second entry is chosen to be \texttt{False}, this information about the transition rates is omitted. Otherwise, the transition rates are given to a number of decimal places determined by the third entry. The default positions of these transition rates are $\frac{1}{6}$ of the way along the edges, with the number closest to the vertex $x$ in the edge $(x, y)$ being $p_{xy}$, but this can be altered manually to avoid overlapping by specification in the fourth entry of \texttt{display\_options}. For example, if one wishes the $p_{45}$ label to be moved to a position $\frac{1}{4}$ of the way from vertex $4$ to vertex $5$ and the $p_{62}$ label to be moved to a position $\frac{1}{5}$ of the way from vertex $6$ to vertex $2$, 
this fourth entry should be chosen to be 
\texttt{[[4, 5, 1/4, 1/6], [6, 2, 1/5, 1/6]]}. 

This completes the description of the functions and routines in the accompanying Python program to this article.

\medskip

{\bf{Acknowledgement:}} Shiping Liu is supported by the National Key R and D Program of China 2020YFA0713100 and the National Natural Science Foundation of China (No. 12031017). We like to thank the London Mathematical Society for their support of Ben Snodgrass via the Undergraduate Research Bursary URB-2021-02, during which the curvature flow was implemented and which lead to many of the research results presented in this paper. David Cushing is supported by the Leverhulme Trust Research Project Grant number RPG-2021-080.

{\footnotesize
\bibliographystyle{amsalpha}
\bibliography{bib}}

\newcommand{\etalchar}[1]{$^{#1}$}
\providecommand{\bysame}{\leavevmode\hbox to3em{\hrulefill}\thinspace}
\providecommand{\MR}{\relax\ifhmode\unskip\space\fi MR }
\providecommand{\MRhref}[2]{%
  \href{http://www.ams.org/mathscinet-getitem?mr=#1}{#2}
}
\providecommand{\href}[2]{#2}
\begin{thebibliography}{CKL{\etalchar{+}}22b}

\bibitem[BE85]{BE-84}
D.~Bakry and Michel \'{E}mery, \emph{Diffusions hypercontractives},
  S\'{e}minaire de probabilit\'{e}s, {XIX}, 1983/84, Lecture Notes in Math.,
  vol. 1123, Springer, Berlin, 1985, pp.~177--206.

\bibitem[CKL{\etalchar{+}}22a]{CKLMPS-22}
David Cushing, Supanat Kamtue, Shiping Liu, Florentin Münch, Norbert
  Peyerimhoff, and Hugo~Benedict Snodgrass, \emph{{B}akry-\'{E}mery curvature
  sharpness and curvature flow in finite weighted graphs. {I}. {T}heory}, arXiv
  preprint arXiv:2204.10064 (2022).

\bibitem[CKL{\etalchar{+}}22b]{CKLLS-19}
David Cushing, Riikka Kangaslampi, Valtteri Lipiäinen, Shiping Liu, and
  George~W. Stagg, \emph{The graph curvature calculator and the curvatures of
  cubic graphs}, Experimental Mathematics \textbf{31} (2022), no.~2, 583--595.

\bibitem[CKLP22]{CKLP-22}
David Cushing, Supanat Kamtue, Shiping Liu, and Norbert Peyerimhoff,
  \emph{Bakry-\'{E}mery curvature on graphs as an eigenvalue problem}, Calc.
  Var. Partial Differential Equations \textbf{61} (2022), no.~2, Paper No. 62,
  33.

\bibitem[CLP20]{CLP-20}
David Cushing, Shiping Liu, and Norbert Peyerimhoff, \emph{Bakry-\'{E}mery
  curvature functions on graphs}, Canad. J. Math. \textbf{72} (2020), no.~1,
  89--143.

\bibitem[Elw91]{Elw-91}
K.~D. Elworthy, \emph{Manifolds and graphs with mostly positive curvatures},
  Stochastic analysis and applications ({L}isbon, 1989), Progr. Probab.,
  vol.~26, Birkh\"{a}user Boston, Boston, MA, 1991, pp.~96--110.

\bibitem[GHL04]{GHL-04}
Sylvestre Gallot, Dominique Hulin, and Jacques Lafontaine, \emph{Riemannian
  {G}eometry}, third ed., Universitext, Springer-Verlag, Berlin, 2004.

\bibitem[KKRT16]{KKRT-16}
Bo'az Klartag, Gady Kozma, Peter Ralli, and Prasad Tetali, \emph{Discrete
  curvature and abelian groups}, Canad. J. Math. \textbf{68} (2016), no.~3,
  655--674.

\bibitem[LY10]{LY-10}
Yong Lin and Shing-Tung Yau, \emph{Ricci curvature and eigenvalue estimate on
  locally finite graphs}, Math. Res. Lett. \textbf{17} (2010), no.~2, 343--356.

\bibitem[Pak10]{Pak-10}
Igor Pak, \emph{Lectures on discrete and polyhedral geometry},
  https://www.math.ucla.edu/~pak/geompol8.pdf, 2010.

\bibitem[Sch99]{Schm-99}
Michael Schmuckenschl\"{a}ger, \emph{Curvature of nonlocal {M}arkov
  generators}, Convex geometric analysis ({B}erkeley, {CA}, 1996), Math. Sci.
  Res. Inst. Publ., vol.~34, Cambridge Univ. Press, Cambridge, 1999,
  pp.~189--197.

\bibitem[Sic20]{Sic-20}
Viola Siconolfi, \emph{Coxeter groups, graphs and {R}icci curvature}, S\'{e}m.
  Lothar. Combin. \textbf{84B} (2020), Art. 67, 12.

\bibitem[Sic21]{Sic-21}
\bysame, \emph{Ricci curvature, graphs and eigenvalues}, Linear Algebra Appl.
  \textbf{620} (2021), 242--267.

\end{thebibliography}

\end{document}